\documentclass[11pt]{article}
\def\thetitle{The wired  minimal spanning forest on \\the Poisson-weighted infinite tree}

\usepackage{amsmath,amssymb}
\usepackage{stackrel}
\usepackage{mathrsfs}
\usepackage[usenames,dvipsnames,svgnames,table]{xcolor}
\definecolor{CombinatoricaAqua}{HTML}{00698C}
\definecolor{CombinatoricaBlue}{HTML}{3A3293}
\definecolor{CombinatoricaBrown}{HTML}{66220C}
\definecolor{CombinatoricaRed}{HTML}{DF2A27}
\definecolor{HarvardCrimson}{rgb}{0.6471, 0.1098, 0.1882}

\makeatletter
\let\reftagform@=\tagform@	
\def\tagform@#1{\maketag@@@
	{(\ignorespaces\textcolor{CombinatoricaBrown}{#1}\unskip\@@italiccorr)}}
\renewcommand{\eqref}[1]{\textup{\reftagform@{\ref{#1}}}}
\makeatother

\usepackage[backref=page]{hyperref}
\hypersetup{
	unicode,
	pdfencoding=auto,
	pdfauthor={Asaf Nachmias and Pengfei Tang},
	pdftitle={\thetitle},
	pdfsubject={},
	pdfkeywords={},
	colorlinks=true,
	citecolor=CombinatoricaBlue,
	linkcolor=CombinatoricaAqua,
	anchorcolor=CombinatoricaBrown,
	urlcolor=HarvardCrimson}

\usepackage{amsthm}
\usepackage{thmtools}
\usepackage{bbm}
\usepackage{enumitem}
\usepackage[capitalize]{cleveref}
\Crefname{fact}{Fact}{Facts}
\Crefname{claim}{Claim}{Claims}


\declaretheoremstyle[
spaceabove=\topsep, spacebelow=\topsep,
headfont=\color{CombinatoricaBrown}\normalfont\bfseries,
bodyfont=\itshape,
]{thm}
\declaretheoremstyle[
spaceabove=\topsep, spacebelow=\topsep,
headfont=\color{CombinatoricaBrown}\normalfont\bfseries,
bodyfont=\normalfont,
]{dfn}
\declaretheoremstyle[
spaceabove=0.5\topsep, spacebelow=0.5\topsep,
headfont=\color{CombinatoricaBrown}\normalfont\bfseries,
bodyfont=\normalfont,
]{rmk}

\declaretheorem[style=thm,parent=section]{theorem}
\declaretheorem[style=thm,sibling=theorem]{lemma}

\declaretheorem[style=thm,sibling=theorem]{observation}

\declaretheorem[style=thm,sibling=theorem]{proposition}
\declaretheorem[style=rmk,sibling=theorem]{remark}

\declaretheorem[style=definition,numbered=no]{acknowledgements}
\declaretheorem[style=definition,sibling=theorem]{definition}

\usepackage[nobysame,msc-links]{amsrefs}

\BibSpec{book}{%
	+{}  {\PrintPrimary}                {transition}
	+{,} { \textbf}                     {title} 
	+{.} { }                            {part}
	+{:} { \textit}                     {subtitle}
	+{,} { \PrintEdition}               {edition}
	+{}  { \PrintEditorsB}              {editor}
	+{,} { \PrintTranslatorsC}          {translator}
	+{,} { \PrintContributions}         {contribution}
	+{,} { }                            {series}
	+{,} { \voltext}                    {volume}
	+{,} { }                            {publisher}
	+{,} { }                            {organization}
	+{,} { }                            {address}
	+{,} { \PrintDateB}                 {date}
	+{,} { }                            {status}
	+{}  { \parenthesize}               {language}
	+{}  { \PrintTranslation}           {translation}
	+{;} { \PrintReprint}               {reprint}
	+{.} { }                            {note}
	+{.} {}                             {transition}
	+{}  {\SentenceSpace \PrintReviews} {review}
}

\makeatletter
\renewcommand{\PrintNames@a}[4]{%
	\PrintSeries{\name}
	{#1}
	{}{ and \set@othername}
	{,}{ \set@othername}
	{}{ and \set@othername}
	{#2}{#4}{#3}%
}
\makeatother

\makeatletter
\def\mathcolor#1#{\@mathcolor{#1}}
\def\@mathcolor#1#2#3{%
	\protect\leavevmode
	\begingroup
	\color#1{#2}#3%
	\endgroup
}
\makeatother
\definecolor{Red}{rgb}{0.618,0,0}
\definecolor{Blue}{rgb}{0,0,1}
\definecolor{Green}{rgb}{0,0.298,0}

\newcommand{\eps}{\varepsilon}

\usepackage{sectsty}
\sectionfont{\color{CombinatoricaBrown}}
\subsectionfont{\color{CombinatoricaBrown}}
\subsubsectionfont{\color{CombinatoricaBrown}}

\usepackage{soul}
\soulregister\ref7

\renewcommand{\P}{\mathbb{P}}

\usepackage[textsize=tiny,backgroundcolor=black!10]{todonotes}
\presetkeys{todonotes}{noline}{}

\usepackage{pifont}

\usepackage[a4paper]{geometry}
\title{\thetitle}
\author{Asaf Nachmias \and Pengfei Tang}

\makeatletter
\def\namedlabel#1#2{\begingroup
  #2%
  \def\@currentlabel{#2}%
  \phantomsection\label{#1}\endgroup
}
\makeatother

\usepackage{mleftright}
\mleftright

\newcommand{\din}{\ensuremath{\mathrm{d}}} 

\newcommand{\iact}{\mathcal{I}}

\newcommand{\scom}{\mathscr{C}}

\DeclareMathOperator{\dist}{dist}

\newcommand{\Gb}{\mathcal{G}_\bullet}
\newcommand{\Gs}{\mathcal{G}_\star}
\newcommand{\Gss}{\mathcal{G}_{\star\star}}

\newcommand{\reff}{\mathcal{R}_{\mathrm{eff}}}

\usepackage{nicefrac}
\newtheorem{example}[theorem]{Example}

\newcommand{\be}{\begin{equation}}
	\newcommand{\ee}{\end{equation}}
\newcommand\ba{\begin{align}}
	\newcommand\ea{\end{align}}

\usepackage[refpage,noprefix]{nomencl}

\makenomenclature

\AtEndDocument{\bigskip{\footnotesize%

\par

\textsc{Asaf Nachmias} \par
\textsc{School of Mathematical Sciences, Tel Aviv University, Tel Aviv 69978, Israel} \par
\textit{Email:} 
\href{mailto:asafnach@tauex.tau.ac.il}{\texttt{asafnach@tauex.tau.ac.il}} \par
\addvspace{\medskipamount}

\par

\textsc{Pengfei Tang} \par
\textsc{Center for applied mathematics, Tianjin University, Tianjin, China} \par
\textit{Email:}
\href{mailto:pengfei\_tang@tju.edu.cn}{\texttt{pengfei\_tang@tju.edu.cn}} \par
\addvspace{\medskipamount}

}}

\begin{document}
	\date{}
\maketitle

\begin{abstract}
	We study the spectral and diffusive properties of the wired minimal spanning forest (WMSF) on the Poisson-weighted infinite tree (PWIT). Let  $M$ be the tree containing the root  in the WMSF on the PWIT 	and $(Y_n)_{n\geq0}$ be a simple random walk on $M$ starting from the root. We show that almost surely $M$ has $\mathbb{P}[Y_{2n}=Y_0]=n^{-3/4+o(1)}$ and $\dist(Y_0,Y_n)=n^{1/4+o(1)}$ with high probability. That is, the spectral dimension of $M$ is $\nicefrac{3}{2}$ and its typical displacement exponent is $\nicefrac{1}{4}$, almost surely. These confirm Addario-Berry's predictions \cite{Addario2013local_limit_MST_complete_graphs}. 
\end{abstract}


\section{Introduction and main results}\label{sec:intro:new} 


The minimal spanning tree (MST) is a fundamental combinatorial optimization problem that has attracted the attention of the computer science, combinatorics and probability communities. We refer the reader to the lovely introduction presented in Section 1 of \cite{Addario-Berry_etal2017_scaling_limit}. 

A classical problem of a probabilistic flavour is the MST on the complete graph with random edge weights. Assign to each edge of the complete graph on $n$ vertices an i.i.d.~random positive weights drawn according to an atomless distribution. The MST is then the almost sure unique spanning tree of minimal total edge weights. 
Several central results on the MST in this setting are as follows: Frieze \cite{Frieze1985value_MST} found the limiting value of its total edge weight; Addario-Berry, Broutin and Reed \cite{Addario-Berry_Broutin_Reed2009diamater_MST} proved that its diameter is of order $n^{1/3}$ with high probability; its scaling limit was identified by Addario-Berry, Broutin, Goldschmidt and Miermont \cite{Addario-Berry_etal2017_scaling_limit} and its local weak limit was identified in \cite{Aldous_Steele2004objective_method,Addario2013local_limit_MST_complete_graphs}. In this paper we are concerned mostly with the last result.

In \cite{Addario2013local_limit_MST_complete_graphs}, Addario-Berry identifies the local limit of the MST on the complete graph to be the component of the root of the \emph{wired minimal spanning forest} (WMSF) on the \emph{Poisson-weighted infinite tree} (PWIT), see the precise definitions in \cref{sec: intro: MSF on PWIT}. In fact this result follows from an earlier more general result of Aldous and Steele \cite[Theorem 5.4]{Aldous_Steele2004objective_method}. The latter shows in particular that the local limit of the MST on any regular graph sequence with degree tending to $\infty$ is again the component of the root of the WMSF on the PWIT. Both \cite{Aldous_Steele2004objective_method} and \cite{Addario2013local_limit_MST_complete_graphs} prove convergence results under a \emph{weighted} topology which does not imply convergence in the more common unweighted topology. In \cref{thm: unweight local limit for regular graph sequence} presented in \cref{sec: local limit thm for MSTs} we point out how to obtain such a convergence. For the main results of this paper this is not relevant as we work directly on the limiting object.

Due to the above, the WMSF on the PWIT is a \emph{universal} limiting object and its behaviour provides information about the local structure of the MST on a large class of finite graphs. Addario-Berry \cite{Addario2013local_limit_MST_complete_graphs} showed that the volume growth rate of the component of the root of the WMSF on the PWIT is cubic (up to lower order fluctuations) and predicted that almost surely its spectral dimension is $\nicefrac{3}{2}$ and that the random walk's typical displacement exponent is $\nicefrac{1}{4}$, see the last paragraph of Section 1.1 in \cite{Addario2013local_limit_MST_complete_graphs}. The main results of this paper confirm these predictions.



\subsection{The WMSF on the PWIT}\label{sec: intro: MSF on PWIT}
Let $U$ be the Ulam--Harris tree, that is, the tree with vertex set $V(U)=\bigcup_{k\geq0}\mathbb{N}^k$ ($\mathbb{N}^0=\{\emptyset\}$), in which for each $k\geq 1$ and each vertex $v=(n_1,\ldots,n_k)\in\mathbb{N}^k$ we form an edge between $v$ and its parent $(n_1,\ldots,n_{k-1})$.
\nomenclature[U]{$U$}{The Ulam-Harris Tree}
\nomenclature[Uow]{$(U,\emptyset,W)$}{The Poisson-weighted infinite tree (PWIT)}
For each $v=(n_1,\ldots,n_k)\in V(U)$, let $(w_i,i\geq1)$ be an independent homogeneous rate one Poisson point process on $[0,\infty)$. For each $i\geq1$ we assign the weight $w_i$ to the edge connecting $v$ to its child $(n_1,\ldots,n_k,i)$. We write $W=\{W(e)\colon e\in E(U)\}$ for these weights on $U$ and root the tree at $\emptyset$. The resulting rooted weighted  tree $(U,\emptyset,W)$ is known as the {\bf Poisson-weighted infinite tree} (PWIT).


Given a connected graph $G=(V,E)$ and weight function $w:E\to (0,\infty)$ the {\bf wired minimal spanning forest} (WMSF) is obtained from $G$ by deleting any edge with a label that is maximal in an \emph{extended} cycle it belongs to, that is, either a cycle or two disjoint infinite simple paths starting from a vertex. 

It will be convenient for us to work with an equivalent definition of the WMSF (\cite[Proposition 3.3]{Lyons_Peres_Schramm2006minimal_spanning_forests}). For this definition we first need to recall the invasion percolation on a PWIT. Given a PWIT sample $(U,\emptyset,W)$, for each $u\in V(U)$, consider the following procedure. Let $I_0(u)$ be the graph with vertex set $\{u\}$ and no edges. Let $I_1(u)$ be the graph formed by the lowest edge incident to $u$ and define $I_n(u)$ recursively thereafter by setting $I_{n+1}(u)=I_n(u)\cup\{e\}$, where $e$ is the lowest edge among the edges not in $I_n(u)$ but incident to a vertex in $ I_n(u)$. Finally the tree $T^{(u)}=\bigcup_{n\geq0} I_n(u)$ is called the {\bf invasion percolation cluster} of $u$. For simplicity we will write $T=T^{(\emptyset)}$ for the invasion percolation cluster of the root $\emptyset$.
\nomenclature[T]{$T^{(u)}$}{The invasion percolation cluster of $u$ on the PWIT $(U,\emptyset,W)$; $T=T^{(\emptyset)}$}

By \cite[Proposition 3.3]{Lyons_Peres_Schramm2006minimal_spanning_forests} it follows that the WMSF is distributed as the union of all the trees in $\big\{T^{(u)}\colon u\in V(U)\big\}$. Let $M$ be the connected component of the WMSF containing the root $\emptyset$. 
\nomenclature[M]{$M$}{The tree of $\emptyset$ in the WMSF on the PWIT} We root $M$ at $\emptyset$ and let it inherit its weights from the weight function $W$ of the PWIT. We thus obtain a random rooted weighted tree $(M, \emptyset, W)$. This is the local weak limit of the MST on complete graphs (and regular graphs of large degree) mentioned earlier and is the central object of study in our paper.

\subsection{Main results}\label{sec: intro: main results}
For simple random walk $(Y_n)_{n\geq0} $ on a locally finite, connected graph $G$ starting from $x\in V(G)$, define the transition density $p_n(x,y):=P^x[Y_n=y]/\mathrm{deg}(y)$. \nomenclature[Yn]{$(Y_n)_{n\geq0}$}{Simple random walk on $M$}
\nomenclature[pn(x,y)]{$p_n(x,y)$}{The transition density for a simple random walk on $M$, defined as $p_n(x,y):=\mathbb{P}^x[Y_n=y]/\mathrm{deg}(y)$.}
The spectral dimension of a graph $G$ is defined as 
\[
d_s(G):=-2\lim_{n\to\infty}\frac{\log p_{2n}(x,x)}{\log n} \, ,
\]
whenever the limit exist (if the limit exists it does not depend on the choice of the vertex $x$). 
\nomenclature[ds(G)]{$d_s(G)$}{The spectral dimension of a graph $G$}

\begin{theorem}\label{thm: spectral dimension}
	Almost surely $M$ has spectral dimension $3/2$. Moreover, there exist a constant $\beta>0$ and a random variable $C\in[1,\infty)$ that only depends on $M$, such that  almost surely the transition density of $M$ satisfies the following: 
	\[
	C^{-1} (\log n)^{-\beta} n^{-3/4}\leq p_{2n}(\emptyset,\emptyset)\leq C(\log n)^{\beta} n^{-3/4} \, , 
	\] 
	for all $n\geq 2$.
\end{theorem}

For our next result we write $d=d_G$ for the graph distance metric on $V(G)$, i.e., $d(x,y)$ is the length of the shortest path between $x$ and $y$ in $G$. 
\nomenclature[d=dG]{$d=d_G$}{The unweighted graph distance on $G$}
Our next result states that the typical graph distance from the origin after $n$ steps of the random walk is order $n^{1/4+o(1)}$ and additionally, the maximal distance attained up to time $n$ is of the same order. 

\begin{theorem}\label{thm: diffusive property}  There exist a constant $\beta_2>0$ and a random variable $C_2\in[1,\infty)$ that only depends on $M$, such that  for almost surely  $M$, 
	\[
	P^{\emptyset}\Big[  C_2^{-1}(\log n)^{-\beta_2} n^{1/4}\leq d(\emptyset,Y_n)\leq C_2(\log n)^{\beta_2}n^{1/4}  \Big]\geq 1-\frac{C_2}{\log n} \, ,
	\]
	for all $n\geq 2$. 
\end{theorem}

\begin{remark} \label{rem: other random walk asymptotics}
	It is by now standard that many other random walk estimates can be proved using the same techniques. For instance, it can be shown that $M$ almost surely satisfied that the number of distinct vertices the random walk visits after $n$ steps is $n^{3/4+o(1)}$, the expected time it takes the walker to reach distance $R$ from the origin is $R^{4+o(1)}$. We omit the details.
\end{remark}

Lastly, we improve the correction terms in  \cite[Theorem 1.2]{Addario2013local_limit_MST_complete_graphs} for the volume growth of $M$ and provide bounds that are tight up to powers of $\log$. Let $B_M(\emptyset,r)=\{v\in M\colon d_M(v,\emptyset)\leq r\}$ denote the ball of radius $r$ centred at the root  $\emptyset$. 
\nomenclature[Bmvr]{$B_{M}(v,r)$}{Ball in $M$ centred at $v$ and with radius $r$.}

\begin{theorem}\label{thm: improved bounds on the volume}
	Almost surely $M$ satisfies
	\be\label{eq: improved upper bound on the volume}
	\limsup_{r\to\infty}\frac{|B_M(\emptyset,r)|}{r^3(\log r)^{18} }<\infty
	\ee
	and  
	\be\label{eq: lower bound on the volume}
	\liminf_{r\to\infty}\frac{|B_M(\emptyset,r)|}{r^3/(\log r)^{18} }>0. 
	\ee
\end{theorem}

\begin{remark} We have not tried to optimize the powers of $\log$ in the three theorems above. In fact, it is quite reasonable that the actual fluctuations are powers of $\log \log$ as in \cite{BarlowKumagaiTrees06}.
\end{remark}


\subsection{Outline of the paper}
We give a quick review of invasion percolation and the WMSF on the PWIT in \cref{sec: preliminaries}. More information can be found in \cite{Addario2013local_limit_MST_complete_graphs} and \cite{Addario_Griffiths_Kang2012invasion_percolation_PWIT}. In Section \ref{sec: volume} we improve the existing bounds on volume growth of $M$ and give the proof of Theorem \ref{thm: improved bounds on the volume}. In Section \ref{sec: resistance} we study the effective resistance from the root to a sphere in $M$ giving an almost linear lower bound in the radius of the sphere; the linear upper bound is trivial. Using these estimates, we use some methods and results of Kumagai and Misumi \cite{Kumagai_Misumi2008} to obtain Theorems \ref{thm: spectral dimension} and \ref{thm: diffusive property}. Lastly, we discuss the universality of the limit, the difference between weighted and unweighted local limits, and show how to use apply a theorem of Aldous and Steele \cite{Aldous_Steele2004objective_method}, to obtain that the component of the root in the WMSF on the PWIT is the (unweighted) local limit of the MST on any sequence of finite simple regular graphs with degree tending to infinity, see \cref{thm: unweight local limit for regular graph sequence}.
Throughout the paper, we will use $c$, $c_i$, $C$ and $C_i$ to represent positive finite constants. The values of these constants may vary from line to line and even within a single string of inequalities.
Finally a list of notation is provided just before the bibliography for readers' convenience.

\section{Preliminaries}\label{sec: preliminaries}

In this section we collect and elaborate on some facts obtained in \cite{Addario2013local_limit_MST_complete_graphs} and \cite{Addario_Griffiths_Kang2012invasion_percolation_PWIT}.

\subsection{The invasion percolation cluster $T$}\label{sec: subsec T}
Invasion percolation on the PWIT has been studied in detail  \cite{Addario_Griffiths_Kang2012invasion_percolation_PWIT} and here we only collect a few results from \cite{Addario_Griffiths_Kang2012invasion_percolation_PWIT} and \cite{Addario2013local_limit_MST_complete_graphs}. Recall that $T=T^{(\emptyset)}$ denotes the invasion percolation cluster of the root $\emptyset$ on the PWIT. 


The first fact we use is \cite[Corollaries 5 and 22]{Addario_Griffiths_Kang2012invasion_percolation_PWIT} stating that $T$ is almost surely one-ended, i.e., the removal of any finite set of vertices leaves precisely one infinite component. Next, $T$ has a natural decomposition into \textit{ponds} and \textit{outlets}. See \cite[Introduction]{van_den_Berg_etal2007size_of_pond_IPC} for a nice interpretation.
We say that an edge $e\in E(T)$ is an {\bf outlet} (also called \emph{forward maximal edge} in \cite{Addario2013local_limit_MST_complete_graphs}) if the removal of $e$ separates the root $\emptyset$ from infinity and if, for any other edge $e'\in E(T)$, if the path from $e'$ to $\emptyset$ contains $e$ then $W(e')<W(e)$. In particular all the outlets are on the unique infinite simple path  that represents the end of $T$ emanating from the root $\emptyset$. 
\nomenclature[Pi]{$P_i$}{The $i$-th pond of $T$}
\nomenclature[RiSi]{$(R_i,S_{i+1})$}{The $i$-th outlet of $T$, where $R_i\in V(P_i)$ and $S_{i+1}\in V(P_{i+1})$}
Write $S_1=\emptyset$, and list the outlets of $T$ in increasing order of distance from the root $\emptyset$ as $\big( (R_i,S_{i+1}), i\geq 1 \big)$. The removal of all outlets separates $T$ into an infinite sequence of finite trees $(P_i,i\geq 1)$,  where for each $i\geq1$, $R_i$ and $S_i$ are vertices of $P_i$. These trees $(P_i,i\geq 1)$ are called the {\bf ponds} of $T$.  

For $\lambda>0$, write $\mathrm{PGW}(\lambda)$ for a branching process tree with progeny distribution of $\textnormal{Poisson}(\lambda)$ (with one particle initially). 
\nomenclature[PGW]{$\mathrm{PGW}(\lambda)$}{Galton-Watson tree with $\mathrm{Poisson}(\lambda)$ progeny distribution.} 
Define $\theta(\lambda):=\mathbb{P}\big[|\mathrm{PGW}(\lambda)| =\infty\big]$. 
For $\lambda >1$, write $B_\lambda$ for a random variable whose distribution is given by   
\[
\mathbb{P}[B_\lambda=m]=\frac{\theta(\lambda)}{\theta'(\lambda)}\cdot \frac{e^{-\lambda m}(\lambda m)^{m-1}}{(m-1)!},\,\,m= 1,2,3,\ldots. 
\]
\nomenclature[Blambda]{$B_{\lambda}$}{Truncated, size-biased Borel-Tanner random variable.}
The fact that $\sum_{m=1}^{\infty}\frac{\theta(\lambda)}{\theta'(\lambda)}\cdot \frac{e^{-\lambda m}(\lambda m)^{m-1}}{(m-1)!}=1$ was proved in  \cite[Corollary 29]{Addario_Griffiths_Kang2012invasion_percolation_PWIT}. 

\begin{definition}\label{def: sample IPC on PWIT}
	Construct a rooted weighted infinite one-ended tree as follows. 
	\begin{enumerate}
		\item[(a)] First sample a Markov chain $(X_n)_{n\geq 1}$ taking positive real values as follows:
		\begin{itemize}
			\item sample a random variable $F$ uniformly distributed in $(0,1)$ and let $X_1=\theta^{-1}(F)$;
			
			\item given $X_n$, sample $X_{n+1}$ using 
			\be\label{eq: transition kernel for X_n}
			\mathbb{P}\big[ X_{n+1}\in dy \mid X_n=x\big]=\frac{\theta'(y)dy}{\theta(x)},\,\,\,1<y<x.
			\ee
		\end{itemize}
		
		\item[(b)]  Given $(X_n)_{n\geq1}$, sample independently positive integers $(Z_n)_{n\geq1}$ such that  $Z_n$ is distributed as $B_{X_n}$.
		
		\item[(c)]  Given $(Z_n)_{n\geq 1}$, sample independently finite trees $(\widetilde{P}_n)_{n\geq1}$ such that for each $n\geq 1$, the tree $\widetilde{P}_n$ is  a uniformly random labelled tree with $Z_n$ vertices. 
		
		\item[(d)] Given $(\widetilde{P}_n)_{n\geq1}$, for each $n\geq1$ pick two vertices $\widetilde{R}_n,\widetilde{S}_n$ uniformly at random from the vertex set of $\widetilde{P}_n$, independently of each other.
		Add an edge $\widetilde{e}_n$ between $\widetilde{R}_n,\widetilde{S}_{n+1}$ for each $n\geq1$ and we get an infinite tree $\widetilde{T}$. Let the root $\rho$ of $\widetilde{T}$ be the vertex $\widetilde{S}_1$. 
		
		\item[(e)] For each $n\geq 1$, assign $\widetilde{e}_n=(\widetilde{R}_n,\widetilde{S}_{n+1})$ with weight $\widetilde{W}(\widetilde{e}_n)=X_n$, and for each edge $e\in E(\widetilde{P}_n)$, let its weight $\widetilde{W}(e)$ be uniform distributed on $[0,X_n]$, independently of each other.

	\end{enumerate}

\end{definition}

Combining Theorem 27, Lemma 28 and Corollary 29 and the remark after Theorem 27 in  \cite{Addario_Griffiths_Kang2012invasion_percolation_PWIT} one has the following  description of the invasion percolation cluster $T$. 
\begin{theorem}\label{thm: outlet weight and pond size of T} Let $(T,\emptyset,W)$ be the invasion percolation cluster $T$ rooted at $\emptyset$ with weights $W$ inherited from the PWIT. Then $(T,\emptyset,W)$ is distributed according to \cref{def: sample IPC on PWIT} and in particular, the ponds $(P_n)_{n\geq1}$ and the weights $\big(W((R_n,S_{n+1}))\big)_{n\geq1}$ of $(T,\emptyset,W)$ have the same distribution as $(\widetilde{P}_n)_{n\geq1}$ and $(X_n)_{n\geq1}$ respectively. 
\end{theorem}

\subsection{Poisson Galton--Watson aggregation process and the construction of $M$}\label{sec: pre: construction of M}
Recall that $M$ is the tree containing the root $\emptyset$ in the WMSF on the PWIT. In particular $T$ is a subtree of $M$.
The tree $M$ is also almost surely one-ended (Corollary 7.2 in \cite{Addario2013local_limit_MST_complete_graphs}). So for any $v\in V(M)$, there is a unique infinite simple path $P(v,\infty)$ starting from $v$ in $M$. Write $x(v)=\sup\{W(e)\colon e\in P(v,\infty)\}$ and call it the \emph{activation time} of $v$. 
\nomenclature[X(v)]{$x(v)$}{The largest weight in the unique infinite path in $M$ starting from $v$, called the activation time of $v$.}
If $v\in V(T)$ and $v$ is in the $i$-th pond $P_i$, then by Theorem \ref{thm: outlet weight and pond size of T}  it is easy to see that $x(v)$ equals the weight of the $i$-th outlet.

Since  $T$ and $M$ are one-ended and $T$ is a subtree of $M$, 
the removal of all edges in $T$ (not the vertices) separates $M$ into a forest containing infinitely many finite trees. Each such tree is naturally rooted at some vertex $v\in V(T)$, denoted by $M_v$. We have seen how to sample $(T,\emptyset,W)$ and next we show how to sample $M_v$ for each $v\in V(T)$, conditioned on $(T,\emptyset,W)$. 
\nomenclature[M(v)]{$M_v$}{The component rooted at $v$ in $M\backslash E(T)$ for $v\in V(T)$}

\subsubsection{Poisson Galton--Watson aggregation process}\label{sec:PGWAdef}

%

\nomenclature[Sstar]{$s^*$}{For $s>1$, $s^*=s\big(1-\theta(s)\big)<1$ is the dual parameter of $s$.}
For $s\geq1$, a $\mathrm{PGW}(s)$ tree conditioned on being finite is distributed as $\mathrm{PGW}(s^*)$, where 
\[
s^*=s\big(1-\theta(s)\big)\leq 1. 
\]
This simple fact is well known and can be easily deduced using Proposition 5.28(ii) in \cite{LP2016}. For $s\geq 1$, the number $s^*$ is called the \textit{dual parameter} of $s$ and is the unique number in $(0,1]$ such that
\[
se^{-s}=s^*e^{-s^*}.
\]
As $s$ increases from $1$ to $\infty$, the dual parameter $s^*$ decreases from $1$ to $0$.

\begin{definition}[One-step Poisson Galton--Watson aggregation]\label{def: one step Poisson GW aggregation}
	Given a vertex $v$ and a real number $\lambda>1$, we construct a rooted, weighted random tree $\mathfrak{T}_{v,\lambda}=(\mathfrak{T},v,W)$ as follows, where $W:E(\mathfrak{T})\to(0,\infty)$ is the weight function on edges of $\mathfrak{T}$.
	\begin{enumerate}
		\item 
		Sample a Poisson process on $[\lambda,\infty)$ with rate $(1-\theta(s))$.
		Let $\mathcal{P}=\{\tau_1,\ldots\}$ be the atoms of this Poisson process. Since $1-\theta(s)$ tends to zero exponentially fast (see \eqref{eq: function equation for theta} below),  almost surely $\mathcal{P}$ is a finite set.
		
		\item 
		\begin{itemize}
			\item If $\mathcal{P}$ is empty, then we let $(\mathfrak{T},v,W)$ be the trivial graph with a single vertex $v$ and no edges. 
			
			\item If $\mathcal{P}$ is not empty, then for each $\tau_i\in \mathcal{P}$, 
			\begin{itemize}
				\item sample a $\mathrm{PGW}(\tau_i^*)$ tree $T_i$. Root  $T_i$  at the progenitor and add an edge $e_i$ between $v$ and the root of $T_i$;
				
				\item assign $W(e_i)=\tau_i$ and assign i.i.d.\ uniform $(0,\tau_i)$ weights to the edges of $T_i$; 
				
			\end{itemize}  
			
		\end{itemize}  
		\item Finally 	we set $(\mathfrak{T},v,W)$ to be the rooted weighted random tree obtained after  the above two operations for all $\tau_i\in\mathcal{P}$. The random tree $\mathfrak{T}$ is almost surely a finite tree since  $\mathcal{P}$ is almost surely a finite set.
	\end{enumerate}
	
\end{definition}

\nomenclature[PGWA]{PGWA}{Poisson Galton--Watson aggregation process}
\begin{definition}[Poisson Galton--Watson aggregation (PGWA) process]\label{def: Poisson GW aggregation process}
	Given a vertex $v$ and a real number $\lambda>1$, construct a rooted, weighted tree $(\widetilde{M}_v,v,\widetilde{W})$ in the following way.

	\begin{enumerate}
		\item Start with a trivial rooted, weighted graph $G_0$ with a single vertex $v$ and no edges. 
		Set $\iact_0=\{v\}$ and let $A(v)=\lambda$. (Also call $A(v)$  the activation time of $v$.)
		
		\item 	Given $G_n,\iact_n$, we construct $G_{n+1},\iact_{n+1}$ as follows. 
		\begin{enumerate}
			\item If $\iact_n$ is an empty set, then simply set $G_{n+1}=G_n$ and $\iact_{n+1}=\iact_n$.
			
			\item Otherwise, 
			\begin{itemize}
				\item for each vertex $w$ in $\iact_n$, if $w\neq v$, then define its activation time $A(w)$ to be the largest weight on the unique path from $w$ to $v$ in the tree $G_n$;
				\item for each vertex $w$ in $\iact_n$, sample a rooted, weighted tree $\mathfrak{T}_{w,A(w)}$ using the one-step Poisson Galton--Watson aggregation process in Definition \ref{def: one step Poisson GW aggregation} and attach this tree to $G_n$ by identifying the root of $\mathfrak{T}_{w, A(w)}$ with the vertex $w$; 
				\item let $G_{n+1}$ be the rooted weighted graph obtained after  attaching these trees  for all $w\in \iact_n$ and set $\iact_{n+1}=V(G_{n+1})\backslash V(G_n)$. In particular, we always root $G_{n+1}$ at $v$. (Note that $\iact_{n+1}$ is empty if and only if all the trees $\mathfrak{T}_{w,A(w)},\,w\in \iact_n$  are trivial.) 
			\end{itemize}

		\end{enumerate}
		
		\item Finally we let $(\widetilde{M}_v,v,\widetilde{W})=\bigcup_{n\geq 1}G_n$. 
		
	\end{enumerate}
\end{definition}

We denote by $\mathbb{P}_\lambda$ and $\mathbb{E}_\lambda$ the law and corresponding expectation of $\widetilde{M}_v$ sampled by the PGWA process starting from time $\lambda$. 
\nomenclature[PlambdaElambda]{$\mathbb{P}_\lambda$}{The law of a tree sampled from the PGWA process starting from time $\lambda$.} It is not obvious that $\widetilde{M}_v$ is finite. In fact, it was shown in \cite[Proposition 7.3 and Corollary 7.2]{Addario2013local_limit_MST_complete_graphs} that $|\widetilde{M}_v|$ has a finite first moment. In particular almost surely the set $\iact_n$ in the PGWA process is empty for all $n$ sufficiently large. 

The following description of $M$ using the PGWA process was shown in \cite[Section 2.2]{Addario2013local_limit_MST_complete_graphs}.
\begin{theorem}\label{thm: dynamic constrution of M}
	First  sample the rooted weighted tree $(T,\emptyset,W)$ using Theorem \ref{thm: outlet weight and pond size of T}. In particular $x(v)$ is then known for each vertex $v\in V(T)$. Conditioned on $(T,\emptyset,W)$, the trees $\{M_v\colon  v\in V(T)\}$ are independent and each tree $M_v$ has the same distribution as a tree sampled from the PGWA process starting from time $x(v)$. 	
\end{theorem}


The following proposition is a simple application of the recursive nature of the PGWA process and will be used later in the proof of Lemma \ref{lem: finite second moment of M_v}.
\begin{proposition}\label{prop: dominate PGWA tree by PGW tree}
	Suppose $\lambda>1$ and $\widetilde{M}_v$ is a tree with law $\mathbb{P}_\lambda$. 
	\begin{enumerate}
		\item[(a)] The tree $\widetilde{M}_v$ is stochastically dominated by a $\mathrm{PGW}\big(\lambda^*+\int_{\lambda}^{\infty}(1-\theta(y))\din y\big)$ tree.
		\item[(b)]   Fix some $\alpha>\lambda$. Color the vertices of $\widetilde{M}_v$ according to their activation times: if $A(u)\in[\lambda,\alpha]$, color $u$ black; otherwise color it white. Observe that white vertex can only have white descendants and the black vertices form a connected subtree rooted at $v$.
		Moreover
		\begin{itemize}
			\item the black subtree of $\widetilde{M}_v$ is stochastically dominated by a $\mathrm{PGW}\big(\lambda^*+\int_{\lambda}^{\alpha}(1-\theta(y))\din y\big)$ tree;
			
			\item conditioned on the black subtree of $\widetilde{M}_v$, the white parts are distributed as i.i.d.\ trees with law $\mathbb{P}_{\alpha}$ rooted at each of the black vertices (except the roots of these i.i.d.\ trees are colored black). 
		\end{itemize}
		
	\end{enumerate} 
	
\end{proposition}
\begin{proof}
	Observe that in the PGWA process, when we use the one-step PGWA process to sample a $\mathrm{PGW}(\tau^*)$ tree, all the vertices in this $\mathrm{PGW}(\tau^*)$ tree have the same activation time $\tau$. 
	For part (a),  first observe that the number of children of $v$ in $\widetilde{M}_v$ has  $\mathrm{Poisson}\big(\int_{\lambda}^{\infty}(1-\theta(y))\din y\big)$ distribution. For any other vertex $u\in V(\widetilde{M}_v)$, the number of children of $u$ in $\widetilde{M}_v$ comes from two parts: one part from the random $\mathrm{PGW}(A(u)^*)$ tree containing $u$ and another part from an independent $\mathrm{Poisson}\big(\int_{A(u)}^{\infty}(1-\theta(y))\din y\big)$ random variable. Since $A(u)\geq \lambda$ (whence $A(u)^*\le\lambda^*$) and the sum of independent Poisson random variables is still Poisson, one has that $\widetilde{M}_v$ is stochastically dominated by a $\mathrm{PGW}\big(\lambda^*+\int_{\lambda}^{\infty}(1-\theta(y))\din y\big)$ tree.
	
	For part (b), similar to part (a) one has that the black subtree is stochastically dominated by a $\mathrm{PGW}\big(\lambda^*+\int_{\lambda}^{\alpha}(1-\theta(y))\din y\big)$ tree. Indeed  the number of black children of $v$ has $\mathrm{Poisson}\big(\int_{\lambda}^{\alpha}(1-\theta(y))\din y\big)$ distribution. For any other black vertex $u\in V(\widetilde{M}_v)$, the number of black children of $u$ in $\widetilde{M}_v$ comes from two parts: one part from the random $\mathrm{PGW}(A(u)^*)$ tree containing $u$ (the vertices of this  $\mathrm{PGW}(A(u)^*)$ tree all have the same activation time $A(u)\in[\lambda,\alpha]$ so are all colored black) and another part from an independent $\mathrm{Poisson}\big(\int_{A(u)}^{\alpha}(1-\theta(y))\din y\big)$ random variable.

	To derive the conditional distribution of the white part of $\widetilde{M}_v$ given the black subtree, we utilize the independence property of a Poisson point process on disjoint intervals. Specifically, for the root $v$, we consider the activation times of its children that fall within the intervals $[\lambda,\alpha]$ and $(\alpha,\infty)$. Let $\{\tau_1,\ldots,\tau_b\}$ represent the activation times within $[\lambda,\alpha]$, and $\{ \hat\tau_1,\ldots,\hat\tau_w \}$ denote the activation times within $(\alpha,\infty)$ (note that $b$ and $w$ can be zero if no activation times fall within the respective intervals). We observe that $\{\tau_1,\ldots,\tau_b\}$ and $\{ \hat\tau_1,\ldots,\hat\tau_w \}$ are independent and are distributed as Poisson point processes with rate $1-\theta(y)$ on the intervals $[\lambda,\alpha]$ and $(\alpha,\infty)$, respectively.
	
	In the PGWA process, we proceed by sampling independent $\mathrm{PGW}(\tau^*)$ trees for each $\tau\in \{\tau_1,\ldots,\tau_b\}\cup \{ \hat\tau_1,\ldots,\hat\tau_w \}$ to obtain $\iact_1$. For every $\tau\in \{\tau_1,\ldots,\tau_b\}$, the vertices in the corresponding $\mathrm{PGW}(\tau^*)$ trees have the same activation time $\tau\in [\lambda,\alpha]$, thus they are colored black. Similarly, for each $\tau\in \{ \hat\tau_1,\ldots,\hat\tau_w \}$, the corresponding $\mathrm{PGW}(\tau^*)$ trees are colored white. Continuing the PGWA process, we find that the vertex $v$, along with these $\mathrm{PGW}(\tau^*)$ trees for $\tau\in \{ \hat\tau_1,\ldots,\hat\tau_w \}$ and their descendants, forms a $\mathbb{P}_{\alpha}$ tree rooted at $v$. Importantly, this tree is independent of the other black children of $v$ and their descendants.
	
	If $b\neq 0$, we can apply a similar analysis for each vertex $u$ in the $\mathrm{PGW}(\tau^*)$ trees associated with $\tau\in\{\tau_1,\ldots,\tau_b\}$, where $A(u)$ assumes the role of $A(v)=\lambda$. By doing so, we can establish that $u$, together with its white children and the descendants of these white children, also forms a $\mathbb{P}_{\alpha}$ tree rooted at $u$. We can repeat this analysis iteratively until we exhaust all black vertices, thus obtaining the desired conditional distribution.
\end{proof}

Lastly we record here a useful  observation.

\begin{observation}\label{obser: stochastic domination of PGWA trees}
	For any $\lambda_2>\lambda_1>1$, a tree constructed by the PGWA process starting from time $\lambda_2$ is stochastically dominated by a tree constructed by the PGWA process starting from time $\lambda_1$.
\end{observation}

\subsubsection{First moment of the size of $M_v$}\label{sec: pre: expected volume of M_v} 

Write $|M_v|$ for the number of vertices in $M_v$. For $z>1$, consider the subgraph of $M_v$ induced by those vertices $w\in V(M_v)$ with activation time $x(w)\leq z$ and write it as $M_v(z)$. Observe that $M_v(z)$ is a subtree since every  vertex  $w\in V(M_v)$ with  $x(w)\leq z$ can be connected to $v$ by a path lying in this subgraph. Write $|M_v(z)|$ for the size of this subtree. 

Given $u\in V(M)$, for $k\geq0$ we say that $u$ has \emph{level} $k$ in $M$ if  these are $k+1$ distinct activation times on the shortest path from $u$ to $T$. That is to say, if $u\in V(T)$, then $u$ has level zero; if $u\in V(M_v)\backslash\{v\}$ for some $v\in V(T)$ and if we sample $M_v$ using the PGWA process starting from time $x(v)$, then $u$ has level $k$ if $u\in \iact_k$.   Write $M_v^k(z)$ for the set of vertices in $M_v(z)$ with level $k$ and write $|M_v^k(z)|$ for its size. In particular, $M_v^0(z)=\{v\}$ for $z\geq x(v)$ and if we sample $M_v$ using the PGWA process, then $M_v^k(\infty)=\iact_k$. By formula (7.1) in \cite{Addario2013local_limit_MST_complete_graphs}, for all $k\geq0,z>\lambda>1$ one has that
\be\label{eq: mean of level k of M_v}
\mathbb{E}\big[ |M_v^k(z)| \mid T, v\in V(T), x(v)=\lambda\big]=n_k(\lambda,z),
\ee 
where
\be\label{eq: n_k's definition}
n_k(\lambda,z)=\int_{\lambda<x_1<\cdots<x_k<z}\prod_{i=1}^{k}\frac{1-\theta(x_i)}{1-x_i^*}~\din x_1\ldots \din x_k\,,
\ee
and we use the convention that $n_0(\lambda,z)=1$. By  \eqref{eq: mean of level k of M_v} and  $|M_v|=\sum_{k=0}^{\infty}|M_v^k(\infty)|$, we have
\be\label{eq: expectation of size of M_v}
\mathbb{E}\big[ |M_v|\mid T, v\in V(T), x(v)=\lambda \big]=\sum_{k=0}^{\infty}n_k(\lambda,\infty).
\ee

\subsection{Several important functions}
Recall that $\theta(x)=\mathbb{P}[|\mathrm{PGW}(x)|=\infty]$. It is well known that $\theta(x)=0$ if $x\in [0,1]$ and $\theta(x)>0$ if $x\in(1,\infty)$. We first collect a few properties of the function $\theta$.
\begin{lemma}\label{lem: property of theta}
	The function $\theta(x)$ has the following properties:
	\begin{enumerate}
		%
		\item $\theta(x)$ is concave and infinitely differentiable on $(1,\infty)$.
		
		\item $\theta(x)$ satisfies the equation
		\be\label{eq: function equation for theta}
		1-\theta(x)=e^{-x\theta(x)}. 
		\ee
		
		
		\item One has that 
		\be\label{eq: approximation of theta near one}
		\theta(x)=2(x-1)-\frac{8}{3}(x-1)^2+o\big((x-1)^2\big) \, ,
		\ee
		as $x \to 1^+$.
	\end{enumerate}
	
\end{lemma}

The proof is simple and is omitted (the interested reader can refer to Corollary 3.19 in \cite{Hofstad2017book_volume1}. Note that  \eqref{eq: approximation of theta near one} is a slight refinement of (3.6.27) there).

\begin{lemma}\label{lem: s-dual}
	For $s=1+\epsilon>1$ one has that 
	\be\label{eq: s star approximation}
	s^*=(1+\epsilon)^*=1-\epsilon+\frac{2}{3}\epsilon^2+O(\epsilon^3) \, ,
	\ee
	as $\epsilon \to 0^+$. 	For $x\in(1,\infty)$, one has that
	\be\label{eq: function f approximation}
	\frac{1-\theta(x)}{1-x^*}=\frac{1}{x-1}-\frac{4}{3}+O\big((x-1)\big).
	\ee
	Furthermore, the function $x\mapsto \frac{1-\theta(x)}{1-x^*}$ decays to zero exponentially fast as $x$ tends to infinity.
\end{lemma}
\begin{proof}
	The  approximations \eqref{eq: s star approximation} and \eqref{eq: function f approximation} can be easily deduced from \eqref{eq: approximation of theta near one} and the fact that $s^*=s\big(1-\theta(s)\big)$. For the exponential decay of the function $x\mapsto \frac{1-\theta(x)}{1-x^*}$ as $x\to\infty$,   notice that $x^*\leq 2^*<1$ and $\theta(x)\ge \theta(2)>0$ for $x\ge 2$, so one has the desired exponential decay:
	\[
	\frac{1-\theta(x)}{1-x^*}\stackrel{\eqref{eq: function equation for theta}}{=}\frac{e^{-x\theta(x)}}{1-x^*}\leq \frac{e^{-x\theta(2)}}{1-2^*}.\qedhere
	\]
	%
\end{proof}

\section{Volume growth}\label{sec: volume}
\subsection{Upper bound on the volume growth}
Our goal here is to provide the proof of the upper bound \eqref{eq: improved upper bound on the volume} from Theorem \ref{thm: improved bounds on the volume}. We first bound the conditional first moment of $|M_v|$ up to a multiplicative constant (this improves \cite[Proposition 7.3]{Addario2013local_limit_MST_complete_graphs}).
\begin{proposition}\label{prop: expected size of M_v}
	For any $\lambda_0>1$, there exist  constants $c>0$ and $C<\infty$ depending only on $\lambda_0$ such that for all $\lambda\in(1,\lambda_0)$, 
	\[
	\frac{c}{\lambda-1}\leq 	\mathbb{E}\big[|M_v|\mid T, v\in V(T), x(v)=\lambda]\stackrel{\eqref{eq: expectation of size of M_v}}{=}\sum_{k\geq0}{n_k(\lambda,\infty)}\leq \frac{C}{\lambda-1}.
	\]
	
\end{proposition}
\begin{proof} In what follows all constants $c,C$ depend only on $\lambda_0$ and may change from line to line. As in the proof of Proposition 7.3 in \cite{Addario2013local_limit_MST_complete_graphs}, 
	for $1<\lambda< \lambda_0$, using symmetry one has that
	\begin{eqnarray*}
		n_k(\lambda,\lambda_0)&\stackrel{\eqref{eq: n_k's definition}}{=}& \int_{\lambda<x_1<\cdots<x_k<\lambda_0}\prod_{i=1}^{k}\frac{1-\theta(x_i)}{1-(x_i)^*}~\din x_1\ldots \din x_k\nonumber\\
		&=&\frac{1}{k!} \int_{\lambda<x_i<\lambda_0,i=1,\ldots,k}\prod_{i=1}^{k}\frac{1-\theta(x_i)}{1-(x_i)^*}~\din x_1\ldots \din x_k=\frac{1}{k!}\bigg[ \int_{\lambda<x<\lambda_0}\frac{1-\theta(x)}{1-x^*}~\din x  \bigg]^k.
	\end{eqnarray*}	
	Hence 
	\[
	\sum_{k\geq0}n_k(\lambda,\lambda_0)=\exp\Big( \int_{\lambda<x<\lambda_0}\frac{1-\theta(x)}{1-x^*}~\din x \Big).
	\]
	
	By \eqref{eq: function f approximation} for all $\lambda\in(1,\lambda_0)$ we have, 
	\be\label{eq: integration of f from lambda to lambda_zero}
	\log\frac{1}{\lambda-1}-C\leq \int_{\lambda<x<\lambda_0}\frac{1-\theta(x)}{1-x^*}~\din x\leq\log\frac{1}{\lambda-1}+C \, . 
	\ee
	Hence 
	\be\label{eq: n_k(lambda,lambda_0)}
	\frac{c}{\lambda-1}\leq \sum_{k\geq0}n_k(\lambda,\lambda_0)\leq \frac{C}{\lambda-1} \, ,
	\ee
	for all $\lambda\in(1,\lambda_0)$. 
	By (7.2)  in \cite{Addario2013local_limit_MST_complete_graphs} we have
	\be\label{eq: n_k(lambda_0, infty)}
	1\leq \sum_{k\geq0}n_k(\lambda_0,\infty)\leq C \, .
	\ee
	Also by splitting according to which $x_i's$ in the integral in \eqref{eq: n_k's definition} are at most $\lambda_0$ it is straightforward to obtain  that 
	\[
	\sum_{k\geq0}{n_k(\lambda,\infty)}=\sum_{k\geq0}{n_k(\lambda,\lambda_0)}\cdot \sum_{k\geq0}{n_k(\lambda_0,\infty)}.
	\]
	So by multiplying \eqref{eq: n_k(lambda,lambda_0)} and \eqref{eq: n_k(lambda_0, infty)} one gets the desired conclusion.
\end{proof}

\subsubsection{Some estimates from \cite{Addario2013local_limit_MST_complete_graphs} about the invasion percolation cluster $T$} \label{sec: subsec estimates from Louigi}

Recall the definitions and construction of $T$ provided in \cref{sec: subsec T}. Let $X_n$ be the weight of the $n$-th outlet $(R_n,S_{n+1})$ in $T$ and recall that $R_n\in V(P_n)$ is the endpoint of the outlet that is closer to the root $\emptyset$. 
For $z>1$, let $I(z):=\min\{ i\colon  X_i \leq z\}$.
\nomenclature[Xn]{$X_n$}{The weight of the $n$-th outlet $(R_n,S_{n+1})$.}
\nomenclature[Iz]{$I(z)$}{The minimal index $n$ such that $X_n\leq z$.}
\begin{lemma}[Lemma 4.3 of \cite{Addario2013local_limit_MST_complete_graphs}]\label{lem: lemma 4.3 of Louigi}
	There exists constants $c>0$ and $C<\infty$ such that for all $r>1$ and $x>1$,
	\[
	\mathbb{P}\Big[\sum_{i\colon X_i>1+\frac{1}{r}}{|V(P_i)|}>xr^2\Big]\leq C\log r\cdot e^{-c\sqrt{x}}.
	\]
\end{lemma}

Let $y_0>0$ be such that $\theta'(1+y_0)=1$; since the function $\theta$ is concave on $[1,\infty)$ and $\theta'(x)\to2$ as $x\downarrow1$ and $\theta'(x)\to0$ as $x\to \infty$, such $y_0$ is unique. Also by  concavity of the function $\theta$ and since $\theta(1)=0$ one has that $1>\theta(1+y_0)=\theta(1+y_0)-\theta(1)\geq y_0\theta'(1+y_0)=y_0$.
The following lemma is a lower tail bound on $d(\emptyset,R_{I(1+\frac{1}{r})})$ and on $X_{I(1+\frac{1}{r})}$.
\begin{lemma}[Lemma 4.4 of \cite{Addario2013local_limit_MST_complete_graphs}]\label{lem: lemma 4.4 of Louigi}
	There exists $C<\infty$ such that for all $r>\frac{1}{y_0}>1$ and all $x>0$, one has that 
	\[
	\mathbb{P}\big[d(\emptyset,R_{I(1+\frac{1}{r})})<xr\textnormal{ or }X_{I(1+\frac{1}{r})}<1+\frac{x}{r} \big]\leq Cx^{2/3}.
	\]
	(The inequality $X_{i_0}<x/r$  in \cite[Lemma 4.4]{Addario2013local_limit_MST_complete_graphs} is a typo and should be $X_{i_0}<1+x/r$.)
\end{lemma}

\subsubsection{Proof of the upper bound \eqref{eq: improved upper bound on the volume} from Theorem \ref{thm: improved bounds on the volume} }

The proof of \eqref{eq: improved upper bound on the volume} is the same as the proof of the upper bound from Theorem 1.2 in \cite{Addario2013local_limit_MST_complete_graphs}, we just replace the use of \cite[Proposition 7.3]{Addario2013local_limit_MST_complete_graphs} with Proposition \ref{prop: expected size of M_v}. We provide the details for completeness. 

For $j\geq1$, let $E_{j}$ denote the event that 
\begin{align*}
	X_{I(1+\frac{1}{2^j})}\leq 1+\frac{1}{j^22^j} \quad \textnormal{or} \quad d\big(\emptyset, R_{I(1+\frac{1}{2^j})}\big)\leq \frac{2^j}{j^2} \quad
	\textnormal{or} \quad \sum_{ \{i\colon X_i>1+1/(j^22^j)\} } |V(P_i)|\geq j^82^{2j} \, . 
\end{align*}
By Lemma \ref{lem: lemma 4.3 of Louigi} and \ref{lem: lemma 4.4 of Louigi} we have $\mathbb{P}[E_{j}]\leq Cj^{-4/3}$ for all $j\geq 1$ for some $C$. Write $L:=\sup\{j\colon E_{j} \textnormal{ occurs}\}$ or $L=0$ if the set is empty.
By Borel-Cantelli $L<\infty$ almost surely. The following equality is due to \cref{thm: dynamic constrution of M} (see also (7.3) of \cite{Addario2013local_limit_MST_complete_graphs}) 
\begin{align}\label{eq: inequality (7.3) in Louigi}
	&\mathbb{E}\Big[ \sum_{\{i\colon X_i>1+1/(j^22^j)\}} \sum_{v\in V(P_i)} |M_v| \,\big| \, \big((X_n,P_n),n\geq 1\big) \Big]\nonumber\\
	&=\sum_{\{i\colon X_i>1+1/(j^22^j)\}}\Big(|V(P_i)|\cdot \sum_{k\geq0} n_k(X_i,\infty)\Big) \nonumber\\
	&\leq \sum_{k\geq0} n_k\big(1+1/(j^22^j),\infty\big)\cdot \sum_{\{i\colon X_i>1+1/(j^22^j)\}}|V(P_i)|
\end{align}

For fixed $j\geq 1$, if $j>L$, then  by definition of $E_{j}$ and $L$ one has that
\[
B_M\big(\emptyset,\frac{2^j}{j^2}\big)\subset \bigcup_{\{i\colon X_i>1+1/(j^22^j)\}}\bigcup_{v\in V(P_i)}M_v
\]
and
\[
\sum_{\{i\colon X_i>1+1+1/(j^22^j)\}} |V(P_i)|\leq j^82^{2j}. 
\]
Applying \eqref{eq: inequality (7.3) in Louigi} one has that
\begin{eqnarray}\label{eq: analogue of inequality (7.4) in Louigi}
	\mathbb{E}\big[|B_M(\emptyset, {2^j \over j^2})| \mid j>L  \big]
	&\leq& j^82^{2j}\cdot \sum_{k\geq0} n_k\big(1+1/(j^22^j),\infty\big)\nonumber\\
	&\leq&Cj^82^{2j}\cdot j^22^j=Cj^{10}2^{3j},
\end{eqnarray}
where the second inequality is due to Proposition \ref{prop: expected size of M_v}. 

If $\limsup_{r\to\infty}|B_M(\emptyset,r)|/\big(r^3(\log r)^{18}\big)\geq 1$ then there exist infinitely many $j\in \mathbb{N}$ such that
\[
|B_M(\emptyset,2^j/j^2)|\geq \big(2^j/j^2\big)^3\cdot j^{17.5}.
\]
Therefore by \eqref{eq: analogue of inequality (7.4) in Louigi} and the conditional Markov inequality one has that 
\begin{align*}
	\mathbb{P}\bigg[ |B_M(\emptyset,2^j/j^2)|\geq \big(2^j/j^2\big)^3\cdot j^{17.5} \textnormal{ for infinitely many }j\in\mathbb{N}\bigg]\\
	\leq \mathbb{P}[L>l]+\sum_{j>l}  \mathbb{P}\Big[ |B_M(\emptyset,2^j/j^2)|\geq \big(2^j/j^2\big)^3\cdot j^{17.5} \mid j>L\Big]\\
	\leq \mathbb{P}[L>l]+\sum_{j>l} \frac{Cj^{10}2^{3j}}{\big(2^j/j^2\big)^3\cdot j^{17.5} },
\end{align*}
and since $L$ is almost surely finite and the sum is convergent, the latter can be made arbitrarily small by choosing $l$ large. It follows that 
\[
\mathbb{P}\Big[\limsup_{r\to\infty}|B_M(\emptyset,r)|/\big(r^3(\log r)^{18}\big)\geq 1  \Big]=0,
\]
which then establishes the desired upper bound on volume growth. \qed

\subsection{Height of $M_v$}

For the proof of the lower bound \eqref{eq: lower bound on the volume} and the  effective resistance estimate in Section \ref{sec: resistance}, we need an estimate on the height of a tree sampled by the  PGWA process. Recall that for $v\in V(T)$, the number $x(v)$ is the largest weight on the unique infinite simple path in $T$ starting from $v$ and that the tree  $M_v$ is the component of $v$ in graph $\big(V(M),E(M)\backslash E(T)\big)$. 
Denote by $\mathrm{Height}(M_v)$ the height of the subtree $M_v$, i.e., $\mathrm{Height}(M_v):=\max\{d(x,v)\colon x\in V(M_v) \}$. 
\begin{proposition}\label{prop: height of M_v}
	
	There exist  constants $r_0,C_1,C_2\geq 1$ such that for all $r\geq r_0$,
	\be\label{eq: a upper bound for the height of M_v}
	\mathbb{P}\Big[ \mathrm{Height}(M_v)\geq C_1r(\log r)^2 \,\,\Big|\,\,T, v\in V(T), x(v)\geq 1+\frac{1}{r}\Big]\leq \frac{C_2}{r^{3}}.
	\ee
\end{proposition}
\begin{proof}
	Recall that $\mathbb{P}_\lambda$ denotes the law of a random tree $\widetilde{M}_v$ sampled by the PGWA process starting from time $\lambda$.
	By Observation \ref{obser: stochastic domination of PGWA trees} and Theorem \ref{thm: dynamic constrution of M},  it suffices to show that 
	\[
	\mathbb{P}_{1+\frac{1}{r}}\Big[ \mathrm{Height}(\widetilde{M}_v)\geq C_1r(\log r)^2 \Big]\leq \frac{C_2}{r^{3}},
	\]
	for all $r\geq r_0$ for some $r_0,C_1,C_2\geq 1$. Write 
	\[
	\beta=\beta(r)=\int_{1+\frac{1}{r}}^{\infty}\frac{1-\theta(x)}{1-x^*}~\din x\,.
	\]
	Then
	\[
	n_k(1+\frac{1}{r},\infty)\stackrel{\eqref{eq: n_k's definition}}{=}\frac{\beta^k}{k!}.
	\]
	Fix an arbitrary $\lambda_0>1$ as in Proposition \ref{prop: expected size of M_v}. By Lemma \ref{lem: s-dual} the function $x\mapsto \frac{1-\theta(x)}{1-x^*}$ decays to zero exponentially fast as $x\to \infty$, so $\int^{\infty}_{\lambda_0}\frac{1-\theta(x)}{1-x^*}~\din x<\infty$. 
	Using this and \eqref{eq: integration of f from lambda to lambda_zero} one  has that for $r\geq \frac{1}{\lambda_0-1}$,
	\[
	\beta=\int_{1+\frac{1}{r}}^{\lambda_0}\frac{1-\theta(x)}{1-x^*}~\din x\,+\int_{\lambda_0}^{\infty}\frac{1-\theta(x)}{1-x^*}~\din x\,=\log r+O(1). 
	\]

	
	Let $A$ denote the event that in the PGWA process of sampling $\widetilde{M}_v$, there is some $k$ such that $|\iact_k|\geq r^{4}$. By Proposition \ref{prop: expected size of M_v} and Theorem \ref{thm: dynamic constrution of M},
	\be\label{eq: order of size of M_v conditioned on x(v) equals lambda}
	\mathbb{E}_{1+\frac{1}{r}}[|\widetilde{M}_v|]=\sum_{k=0}^{\infty}n_k(1+\frac{1}{r},\infty)\leq C_3r.
	\ee
	If $A$ occurs, then $|\widetilde{M}_v|\geq r^{4}$. By Markov's inequality 
	\[
	\mathbb{P}_{1+\frac{1}{r}}[A]\leq \frac{\mathbb{E}_{1+\frac{1}{r}}[|\widetilde{M}_v|]}{r^{4}}\leq \frac{C_3}{r^{3}}.
	\]
	
	Let $B$ denote the event that  there are at least $3e^2\beta$ many levels  in $\widetilde{M}_v$, i.e., that the set $\iact_{3e^2 \beta}$ is non-empty. By \eqref{eq: mean of level k of M_v} and Theorem \ref{thm: dynamic constrution of M}, 
	for $C=3e^2$, $$\mathbb{E}_{1+\frac{1}{r}}\big[ |\iact_{C\beta}|\big]=n_{C\beta}(1+\frac{1}{r},\infty)=\frac{\beta^{C\beta}}{(C\beta)!}\leq \big(\frac{e}{C}\big)^{C\beta}\frac{1}{\sqrt{2\pi C\beta}}\leq \frac{1}{e^{3\beta}}\leq \frac{C_4}{r^{3}}.$$
	Hence
	\[
	\mathbb{P}_{1+\frac{1}{r}}[B]=\mathbb{P}_{1+\frac{1}{r}}[\iact_{C\beta}\neq \emptyset]\leq \mathbb{E}_{1+\frac{1}{r}}\big[ |\iact_{C\beta}|\big]\leq \frac{C_4}{r^{3}}.
	\]
	
	In the PGWA process of sampling $\widetilde{M}_v$, each tree in $G_{k+1}\backslash G_k$ is distributed as $\mathrm{PGW}(s^*)$ for some  $s\geq 1+\frac{1}{r}$. By \eqref{eq: s star approximation} there exists $r_0>1$ such that for all $r\geq r_0$ and $s\geq 1+\frac{1}{r}$, 
	\[
	s^*\leq (1+\frac{1}{r})^*\leq 1-\frac{1}{2r}. 
	\]
	%
	%
	Since the $h$-th generation of a $\mathrm{PGW}(s^*)$ tree has expected size $(s^*)^h$, one has that
	\[
	\mathbb{P}\big[\mathrm{Height}\big(\mathrm{PGW}(s^*) \big) \geq h\big]\leq (s^*)^h.
	\]
	Thus one can choose $C_5\in(0,\infty)$ such that for all $r\geq r_0$ and $s\geq 1+\frac{1}{r}$, the probability that a $\mathrm{PGW}(s^*)$ tree has height at least $C_5r\log r$ is at most 
	\[
	(s^*)^{C_5r\log r}\leq (1-\frac{1}{2r})^{C_5r\log r}\leq \frac{1}{r^{7}}.
	\]

	Now  fix a large constant $C_1$  such that for all $r\geq r_0$, 
	\[
	\frac{C_1r(\log r)^2}{3e^2\beta}\geq C_5r\log r.
	\]
	Consider the event $\{\mathrm{Height}(\widetilde{M}_v)\geq C_1r(\log r)^2\} \cap A^c\cap  B^c$. Since $\iact_{3e^2\beta}$ is empty on $B^c$, there is some $k\in[0,3e^2\beta-1]$ such that some of the trees in $G_{k+1}\backslash G_k$ has height at least $\frac{C_1r(\log r)^2}{3e^2\beta}\geq C_5r\log r.$ Also on $A^c$, for all $k\in[0,3e^2\beta-1]$ there are at most $|\iact_k|\leq r^{4}$ trees in $G_{k+1}\backslash G_k$, and each of them is distributed as $\mathrm{PGW}(s^*)$ for some $s\geq 1+\frac{1}{r}$, independently of each other. Hence by a union bound one has that 
	\be
	\mathbb{P}_{1+\frac{1}{r}}\Big[\big \{\mathrm{Height}(\widetilde{M}_v)\geq C_1r(\log r)^2 \big\}\cap A^c\cap B^c\Big]\leq r^{4}\cdot \frac{1}{r^{7}}=\frac{1}{r^{3}}.
	\ee
	Combining the above, one has the desired result:
	\begin{eqnarray*}
		&&\mathbb{P}_{1+\frac{1}{r}}\big[ \mathrm{Height}(\widetilde{M}_v)\geq C_1r(\log r)^2 \big]\nonumber\\
		&\leq& \mathbb{P}_{1+\frac{1}{r}}[A]+\mathbb{P}_{1+\frac{1}{r}}[B]+	\mathbb{P}_{1+\frac{1}{r}}\Big[ \big\{\mathrm{Height}(\widetilde{M}_v)\geq C_1r(\log r)^2 \big\}\cap A^c\cap B^c\Big]\nonumber\\
		&\leq& \frac{C_3}{r^{3}}+\frac{C_4}{r^{3}}+\frac{1}{r^{3}}\leq  \frac{C_2}{r^{3}}. \qedhere
	\end{eqnarray*}
\end{proof}

\subsection{Lower bound on the volume growth}

We now proceed towards proving the lower bound in \cref{thm: improved bounds on the volume}.

%
%

\subsubsection{Second moment of the size of $M_v$}
We have seen from Proposition \ref{prop: expected size of M_v} that   $\mathbb{E}[|M_v|\mid T, v\in V(T), x(v)=\lambda]$ is of order $ \frac{1}{\lambda-1}$ for $\lambda$ close to $1$. In the following  we shall show that the second moment is of order at most $\frac{1}{(\lambda-1)^4}$.
\begin{proposition}\label{prop: second moment of M_v}
	There exist constants $\lambda_0>1$ and  $C>0$ such that for all $\lambda\in(1,\lambda_0]$,
	\[
	\mathbb{E}[|M_v|^2\mid T,v\in V(T),x(v)=\lambda]\leq \frac{C}{(\lambda-1)^4}.
	\]
\end{proposition}
\begin{remark}
	There also exist constants $\lambda_0>1$ and $c=c(\lambda_0)>0$ such that for all $\lambda\in(1,\lambda_0]$,
	\[
	\mathbb{E}[|M_v|^2\mid T,v\in V(T),x(v)=\lambda]\geq \frac{c}{(\lambda-1)^4} \, ,
	\]
	however, we will not need this lower bound in this paper, so we omit its (rather easy) proof. 
\end{remark}

We first show that the second moment is finite for every $\lambda>1$.
\begin{lemma}\label{lem: finite second moment of M_v}
	For all $\lambda\in(1,\infty)$ one has that
	\[
	\mathbb{E}[|M_v|^2\mid T, v\in V(T),x(v)=\lambda]<\infty.
	\]
\end{lemma}
\begin{proof}
	By Theorem \ref{thm: dynamic constrution of M} for $v\in V(T)$, the tree $M_v$ given $T$ and $x(v)=\lambda$ is distributed as a random tree $\widetilde{M}_v$ sampled from the PGWA process starting from time $\lambda$. Recall that the law of $\widetilde{M}_v$ is denoted by $\mathbb{P}_\lambda$.
	For simplicity we write
	\[
	h(\lambda):=\mathbb{E}_{\lambda}[|\widetilde{M}_v|^2 ]=\mathbb{E}[|M_v|^2\mid T, v\in V(T),x(v)=\lambda]
	\]
	and
	\[
	\Lambda:=\inf\{\lambda\colon \lambda\in (1,\infty), h(\lambda)<\infty\}.
	\]
	Since $h(\lambda)$ is decreasing in $\lambda$ by Observation \ref{obser: stochastic domination of PGWA trees}, it suffices to show  $\Lambda=1$.
	
	First we show that $\Lambda<\infty$. Let $\widetilde{M}_v$ be sampled from the PGWA process starting from time $\lambda$. By  Proposition \ref{prop: dominate PGWA tree by PGW tree} $\widetilde{M}_v$ is stochastically dominated by a $\mathrm{PGW}\big(\lambda^*+\int_{\lambda}^{\infty}(1-\theta(y))\din y \big)$ tree. Note that   $\lambda^*+\int_{\lambda}^{\infty}(1-\theta(y))\din y<1$ for $\lambda$ sufficiently large. We recall  formula (4.1) from \cite{Addario2013local_limit_MST_complete_graphs}: for $\lambda \in (0,1)$, the first moment and second moment of $|\mathrm{PGW}(\lambda)|$ are 
	\be\label{eq: formula4.1Louigi}
	\mathbb{E}\big[  |\mathrm{PGW}(\lambda)| \big]=\frac{1}{1-\lambda},\quad \mathbb{E}\big[  |\mathrm{PGW}(\lambda)|^2 \big]=\frac{1}{(1-\lambda)^3}.
	\ee 
	Then
	the following inequality holds for all $\lambda$ sufficiently large 
	\[
	h(\lambda)\leq \frac{1}{\Big(1-\big(\lambda^*+\int_{\lambda}^{\infty}(1-\theta(y))\din y\big)\Big)^3}<\infty.
	\]

	Next, we prove that $\Lambda=1$. Suppose the opposite, i.e., that $\Lambda>1$. In this case, we can select a sufficiently small value $\varepsilon\in(0,\Lambda-1)$ such that
	\[
	q(\Lambda,\varepsilon):=(\Lambda-\varepsilon)^*+\int_{\Lambda-\varepsilon}^{\Lambda+\varepsilon}(1-\theta(y))\din y<1.
	\]
	Now, let us examine $h(\Lambda-\varepsilon)$. We sample $\widetilde{M}_v$ from the PGWA process, initiating it at time $\Lambda-\varepsilon$. We assign colors to the vertices of $\widetilde{M}_v$ based on their activation times, as described in Proposition \ref{prop: dominate PGWA tree by PGW tree}: for each $u\in V(\widetilde{M}_v)$, if $A(u)\in [\Lambda-\varepsilon,\Lambda+\varepsilon]$, we color it black (including vertex $v$); otherwise, we color it white.
	Writing $N$ for the number of black vertices, by Proposition \ref{prop: dominate PGWA tree by PGW tree} the size  $|\widetilde{M}_v|$ is distributed as $
	\sum_{j=1}^{N}Y_i$,
	where $Y_1,Y_2,\ldots$ are i.i.d.\ r.v.'s independent of $N$ and are distributed as the size of a random tree with law   $\mathbb{P}_{\Lambda+\varepsilon}$. 
	Note that $\mathbb{E}[Y_iY_j]\leq \mathbb{E}[Y_i^2]=h(\Lambda+\varepsilon)<\infty$.  Using Proposition \ref{prop: dominate PGWA tree by PGW tree}  and  \eqref{eq: formula4.1Louigi} once again one has that
	\[
	\mathbb{E}_{\Lambda-\varepsilon}[N^2]\leq \frac{1}{\big(1-q(\Lambda,\varepsilon)\big)^3}<\infty.
	\]
	Therefore one obtains that
	\begin{eqnarray*}
		h(\Lambda-\varepsilon)&=&\mathbb{E}_{\Lambda-\varepsilon}\big[|\widetilde{M}_v|^2\big]= \mathbb{E}_{\Lambda-\varepsilon}\Big[\sum_{i=1}^{N}\sum_{j=1 }^{N}Y_iY_j\Big]\nonumber\\
		&\leq& h(\Lambda+\varepsilon)\cdot \frac{1}{\big(1-q(\Lambda,\varepsilon)\big)^3} <\infty.
	\end{eqnarray*}
	This contradicts  the definition of $\Lambda$. Hence one must have $\Lambda=1$. 
\end{proof}

\begin{lemma}\label{lem: integration equation for h(lambda)}
	For $\lambda>1$, let $\widetilde{M}_v$ be a random tree with law of $\mathbb{P}_{\lambda}$. Write $h(\lambda):=\mathbb{E}_{\lambda}[|\widetilde{M}_v|^2]$ as before  and define $g(\lambda):=\mathbb{E}_{\lambda}[|\widetilde{M}_v| ]$. 
	Then one has the following equality 
	\[
	h(\lambda)=g^2(\lambda)+\int_{\lambda}^{\infty}\frac{1-\theta(y)}{1-y^*}h(y)\din y+\int_{\lambda}^{\infty}\bigg( \frac{1-\theta(y)}{(1-y^*)^3}-\frac{1-\theta(y)}{1-y^*} \bigg)g^2(y)\din y.
	\]
\end{lemma}
\begin{proof}
	
	Let $L$ denote the number of children of $v$ in $\widetilde{M}_v$ and list its children as $v_1,\ldots,v_L$. Let $y_i$ be the weight of the edge $(v,v_i)$. 
	From the definition of the PGWA process, conditioned on $L$ and $\{y_1,\ldots,y_L\}$, one can sample $\widetilde{M}_v$ as follows: 
	\begin{itemize}
		\item Firstly, independently for each $i\in \{1,\ldots, L\}$, let  $\mathfrak{T}_i$ be a $\mathrm{PGW}(y_i^*)$ rooted at $v_i$.
		\item Secondly, for each $i\in \{1,\ldots,L\}$ and each vertex $w\in V(\mathfrak{T}_i)$ sample an independent random tree with law $\mathbb{P}_{y_i}$ and attach it to $\mathfrak{T}_i$ at $w$ (by identifying the root of the random tree with $w$).
	\end{itemize} 
	Then
	\[
	|\widetilde{M}_v|\stackrel{(d)}{=}1+\sum_{i=1}^{L} \sum_{j=1}^{A_i} |M_{i,j}(y_i)|,
	\]
	where $1$ corresponds to $v$ itself, and $A_{i}\stackrel{(d)}{=}|\mathrm{PGW}(y_i^*)|$ is the size of the $i$-th $\mathrm{PGW}$ tree $\mathfrak{T}_i$ that is attached to $v$, and $M_{i,j}(y_i)$ are independent random trees with the law of $\mathbb{P}_{y_i}$.
	Hence 
	\begin{eqnarray*}\label{eq: square of M_v's size}
		|\widetilde{M}_v|^2&\stackrel{(d)}{=}&1+\sum_{i=1}^{L} \sum_{j=1}^{A_i} |M_{i,j}(y_i)|^2+2\sum_{i=1}^{L} \sum_{j=1}^{A_i} |M_{i,j}(y_i)|+
		\sum_{i=1}^{L} \sum_{j=1}^{A_i} \sum_{j'\neq j} |M_{i,j}(y_i)|\cdot |M_{i,j'}(y_i)|\nonumber\\
		&&+\sum_{i=1}^{L} \sum_{j=1}^{A_i} \sum_{i'\neq i} \sum_{j'=1}^{A_{i'}} |M_{i,j}(y_i)|\cdot|M_{i',j'}(y_i')|
	\end{eqnarray*}
	Taking expectation  (still conditioned on $L$ and $\{y_1,\ldots,y_L\}$) and using $\mathbb{E}[A_i]=\frac{1}{1-y_i^*}$ and $\mathbb{E}(A_i^2)=\frac{1}{(1-y_i^*)^3}$ (formula \eqref{eq: formula4.1Louigi}) one has that 
	\begin{eqnarray*}
		\mathbb{E}_{\lambda}\big[|\widetilde{M}_v|^2\mid L,y_1,\ldots,y_L\big]
		&=&1+\sum_{i=1}^{L}\frac{1}{1-y_i^*}h(y_i)+2\sum_{i=1}^{L}\frac{1}{1-y_i^*}g(y_i)\nonumber\\
		&&+\sum_{i=1}^{L}\bigg( \frac{1}{(1-y_i^*)^3}-\frac{1}{1-y_i^*} \bigg)g(y_i)^2+\nonumber\\
		&&+\sum_{i=1}^{L}\sum_{i'\neq i}\frac{1}{1-y_i^*}\cdot \frac{1}{1-y_{i'}^*}g(y_i)g(y_{i'})
	\end{eqnarray*}
	Taking expectation over $L,y_1,\ldots,y_L$ one has that
	\begin{eqnarray}\label{eq: 1.4}
		h(\lambda)&=&1+\int_{\lambda}^{\infty}\frac{1-\theta(y)}{1-y^*}h(y)~\din y+2\int_{\lambda}^{\infty}\frac{1-\theta(y)}{1-y^*}g(y)~\din y\nonumber\\
		&&+\int_{\lambda}^{\infty}\bigg( \frac{1-\theta(y)}{(1-y^*)^3}-\frac{1-\theta(y)}{1-y^*} \bigg)g(y)^2~\din y\nonumber\\
		&&+\int_{\lambda}^{\infty}\int_{\lambda}^{\infty}\frac{1-\theta(y)}{1-y^*}\frac{1-\theta(z)}{1-z^*}g(y)g(z)~\din y\din z
	\end{eqnarray}
	Similar and simpler analysis  gives 
	\be\label{eq: integration equation for g(lambda)}
	g(\lambda)=1+\int_{\lambda}^{\infty}\frac{1-\theta(y)}{1-y^*}g(y)~\din y,
	\ee
	(this can also be verified using $g(\lambda)=\sum_{k=0}^{\infty}n_k(\lambda,\infty)$ and the definition of $n_k(\lambda,\infty)$.) Squaring the identity \eqref{eq: integration equation for g(lambda)} and combining the result with \eqref{eq: 1.4} yields the desired conclusion.
\end{proof}

\begin{proof}[Proof of Proposition \ref{prop: second moment of M_v}] We use the notation $h(\lambda)$ and $g(\lambda)$ from the previous lemma. 
	By Observation \ref{obser: stochastic domination of PGWA trees} and Lemma \ref{lem: finite second moment of M_v} the function $h$ is a decreasing  function from $(1,\infty)$ to $(1,\infty)$. Hence it suffices to show that there exist constants $\lambda_0>1$ and $C\geq 1$ such that for all $i\geq0$
	\be\label{eq: induction step for second momnent of M_v}
	h(x_i)\leq \frac{C}{(x_i-1)^4},
	\ee
	where $(x_i)_{i\geq 0}$ is the sequence converging to $1$ such that $x_0=\lambda_0$ and $x_i = {1 + x_{i-1} \over 2}$. We begin by specifying the choice of $\lambda_0$ and $C$. First by \eqref{eq: s star approximation} there exists $\lambda_0\in(1,\infty)$ close to $1$ such that for any $y\in (1,\lambda_0]$ 
	\[
	1-y^* \geq \frac{3}{4}(y-1) \, .
	\]
	Similarly, the function $g(\lambda)$ is also decreasing on $(1,\infty)$ and since $1-\theta(y)$ decreases exponentially fast to $0$ as $y\to\infty$ (so that $1-y^*\to 1$), we can choose
	\[
	C_1=C_1(\lambda_0):=\int_{\lambda_0}^{\infty}\frac{1-\theta(y)}{(1-y^*)^3}g^2(y)~\din y\leq g^2(\lambda_0)\int_{\lambda_0}^{\infty}\frac{1-\theta(y)}{(1-y^*)^3}~\din y<\infty. 
	\]
	By Proposition \ref{prop: expected size of M_v} there exists a constant $C_2>0$ such that 
	\[
	g(\lambda)\leq \frac{C_2}{\lambda-1},\,\,\forall\,\,\lambda\in(1,\lambda_0].
	\]
	We may further choose a constant $C_3=C_3(\lambda_0)$ sufficiently large such that for all $\lambda\in \big(1,\lambda_0\big]$ one has that
	\begin{eqnarray}\label{eq: choice of a large constant C one}
		&&g^2(\lambda)+\int_{\lambda}^{\infty}\bigg( \frac{1-\theta(y)}{(1-y^*)^3}-\frac{1-\theta(y)}{1-y^*} \bigg)g^2(y)~\din y\nonumber\\
		&\leq& 	g^2(\lambda)+\int_{\lambda}^{\lambda_0}\frac{1}{(1-y^*)^3}\cdot  g^2(y)~\din y+C_1(\lambda_0)
		\nonumber\\
		&\leq& \frac{C_2^2}{(\lambda-1)^2}+\int_{\lambda}^{\lambda_0}\frac{64C_2^2}{27(y-1)^5}~\din y+C_1(\lambda_0)\leq \frac{C_3}{(\lambda-1)^4}.
	\end{eqnarray}
	Finally, by Lemma \ref{lem: finite second moment of M_v}  we can choose a constant  $C=C(\lambda_0)>0$ sufficiently large such that 
	\[
	h(\lambda_0)\leq  \frac{C}{(\lambda_0-1)^4}
	\]
	and 
	\[
	\frac{1}{1-\frac{4}{3}\log 2}\big[C_3+\frac{C}{2^4}\big]\leq C. 
	\]

	Having chosen $\lambda_0$ and $C=C(\lambda_0)$, we now prove \eqref{eq: induction step for second momnent of M_v} by induction. By the requirement on $C$, the inequality \eqref{eq: induction step for second momnent of M_v} holds for $i=0$. Suppose \eqref{eq: induction step for second momnent of M_v} holds for $i-1$ ($i\geq 1$). Then by Lemma \ref{lem: integration equation for h(lambda)} one has that
	\begin{eqnarray*}
		h(x_i)&=&
		g(x_i)^2+\int_{x_i}^{\infty}\bigg( \frac{1-\theta(y)}{(1-y^*)^3}-\frac{1-\theta(y)}{1-y^*} \bigg)g^2(y)~\din y+\int_{x_i}^{\infty}\frac{1-\theta(y)}{1-y^*}h(y)~\din y\nonumber\\
		&\stackrel{\eqref{eq: choice of a large constant C one}}{\leq}& \frac{C_3}{(x_i-1)^4}+\int_{x_i}^{x_{i-1}}\frac{1-\theta(y)}{1-y^*}h(y)~\din y+\int_{x_{i-1}}^{\infty}\frac{1-\theta(y)}{1-y^*}h(y)~\din y
	\end{eqnarray*}
	By Lemma \ref{lem: integration equation for h(lambda)} we bound the last term 
	\[
	\int_{x_{i-1}}^{\infty}\frac{1-\theta(y)}{1-y^*}h(y)~\din y\leq h(x_{i-1}).
	\]
	Using this and the monotonicity of $h$ one has that 
	\[
	h(x_i)\leq \frac{C_3}{(x_i-1)^4}+h(x_i)\int_{x_i}^{x_{i-1}}\frac{1}{1-y^*}~\din y+h(x_{i-1}).
	\]
	By the choice of $\lambda_0$ one has that $1-y^*\geq \frac{3}{4}(y-1)$ on $[x_{i},x_{i-1}]$. Hence 
	\[
	h(x_i)\leq \frac{C_3}{(x_i-1)^4}+h(x_i)\cdot \frac{4}{3}\log2+h(x_{i-1}).
	\]
	By induction hypothesis $h(x_{i-1})\leq \frac{C}{(x_{i-1}-1)^4}=\frac{C}{2^4(x_i-1)^4}$. Therefore one has that 
	\[
	h(x_i)\big[1-\frac{4}{3}\log 2\big]\leq [C_3+\frac{C}{2^4}]\cdot \frac{1}{(x_i-1)^4}.
	\]
	Then by the choice of the large constant $C$, the inequality \eqref{eq: induction step for second momnent of M_v} also holds for $i$.
\end{proof}

\subsubsection{Proof of the lower bound \eqref{eq: lower bound on the volume} from Theorem \ref{thm: improved bounds on the volume}}

We begin by a rough sketch of the proof: 
\begin{enumerate}
	\item[(a)]  First we show that with high probability  the set $F:=\{ v\in V\big(B_T(\emptyset,\frac{r}{2})\big)\colon x(v)\leq 1+\frac{(\log  r)^{6}}{r} \}$ has size at least $\frac{r^2}{(\log r)^{9}}$.
	
	\item[(b)] Conditioned on $T$ and the event above, by Observation \ref{obser: stochastic domination of PGWA trees} and Theorem \ref{thm: dynamic constrution of M} each tree in $\{M_v,v\in F\}$ stochastically dominates a tree with law  $\mathbb{P}_{1+\frac{(\log  r)^{6}}{r}}$. (Recall that $\mathbb{P}_\lambda$ is the law of a tree sampled from the PGWA process starting from time $\lambda$.)
	
	\item[(c)]
	Using Paley--Zygmund inequality one has that with high probability the sum of the sizes of $\frac{r^2}{(\log r)^{9}}$ i.i.d.\ random trees with law  $\mathbb{P}_{1+\frac{(\log  r)^{6}}{r}}$ is at least $\frac{r^3}{(\log r)^{18}}$. Also at the same time all these i.i.d.\ random trees  have height at most $\frac{r}{2}$ (using Proposition \ref{prop: height of M_v}). 
	
	\item[(d)] Using (b) and (c) one has that the inequality  $|B_M(\emptyset,r)|\geq \frac{r^3}{(\log r)^{18}}$ holds with high probability. We then finish the proof by a routine application of Borel--Cantelli lemma. 
\end{enumerate}

We begin with a lemma that corresponds to the item (c).
\begin{lemma}\label{lem: paley-zygmund lower bound}
	For any $r>1$ let $m=m(r)=\frac{r^2}{(\log r)^{9}}$ and $T_1,\ldots,T_m$ be i.i.d.\ random trees with law $\mathbb{P}_{1+\frac{(\log  r)^{6}}{r}}$. Also let $A$ denote the event that no tree in $\{T_1,\ldots,T_m\}$ has height at least $\frac{r}{2}$. 
	Then there exists  constants $0<r_0,C < \infty$ such that for $r\geq r_0$,
	\[
	\mathbb{P}\bigg[  A\cap \Big\{\sum_{i=1}^{m}|V(T_i)|\geq \frac{r^3}{(\log r)^{18}} \Big\} \bigg]\geq 1-\frac{C}{(\log r)^{2}}.
	\]
\end{lemma}
\begin{proof}
	By  Theorem \ref{thm: dynamic constrution of M} and Proposition \ref{prop: expected size of M_v} we have that $\mathbb{E}[\sum_{i=1}^{m}|V(T_i)|]$ is of order $r^3/(\log r)^{15}$.
	Next, by Theorem \ref{thm: dynamic constrution of M} and Proposition \ref{prop: second moment of M_v} we have that $\mathbb{E}[|V(T_i)|^2]$ is at most $O\big(r^4/(\log r)^{24}\big)$. Hence $\textnormal{Var}\big(\sum_{i=1}^{m}|V(T_i)|\big)=m\textnormal{Var}\big(|V(T_i)|\big)=O\big(r^6/(\log r)^{33}\big)$.
	We now put 
	\[
	\delta=\frac{1}{(\log r)^{2}}, 
	\]
	so that when $r$ is large enough we get 
	\begin{eqnarray*}
		\mathbb{P}\Big[ \sum_{i=1}^{m}|V(T_i)|\geq \frac{r^3}{(\log r)^{18}} \Big]
		&\geq&\mathbb{P}\Big[ \sum_{i=1}^{m}|V(T_i)|\geq \delta\mathbb{E}\big[\sum_{i=1}^{m}|V(T_i)|\big] \Big]\nonumber\\
		&\geq& \big(1-\delta\big)^2\frac{\Big[\mathbb{E}\big[\sum_{i=1}^{m}|V(T_i)|\big]\Big]^2}{\mathbb{E}\Big[\big(\sum_{i=1}^{m}|V(T_i)|\big)^2\Big]} 
	\end{eqnarray*}
	where the last bound follows from the  Paley--Zygmund inequality. Since 
	\[
	\frac{\Big[\mathbb{E}\big[\sum_{i=1}^{m}|V(T_i)|\big]\Big]^2}{\mathbb{E}\Big[\big(\sum_{i=1}^{m}|V(T_i)|\big)^2\Big]}=1-\frac{	\textnormal{Var}\big(\sum_{i=1}^{m}|V(T_i)|\big)}{\mathbb{E}\Big[\big(\sum_{i=1}^{m}|V(T_i)|\big)^2\Big]}\geq 1-\frac{	\textnormal{Var}\big(\sum_{i=1}^{m}|V(T_i)|\big)}{\Big[\mathbb{E}\big[\sum_{i=1}^{m}|V(T_i)|\big]\Big]^2},
	\]
	one has that 
	\be\label{eq: Paley-Zygmund application}
	\mathbb{P}\Big[ \sum_{i=1}^{m}|V(T_i)|\geq \frac{r^3}{(\log r)^{18}} \Big]
	\geq 1 - {C \over \log^2 r} \, ,
	\ee
	for some $C<\infty$. Next by Proposition \ref{prop: height of M_v} and the union bound we have
	\[
	\mathbb{P}\Big[\exists\, i\in\{1,\ldots,m\}\textnormal{ such that }\mathrm{Height}(T_i)\geq \frac{r}{2}\Big]\leq \frac{ C_2 m (\log r)^{18}}{r^{3}} \leq \frac{C_3}{(\log r)^2} \, .
	\]
	Combining this with \eqref{eq: Paley-Zygmund application} one obtains the desired conclusion.
\end{proof}

We need another ingredient from \cite{Addario2013local_limit_MST_complete_graphs} which gives a upper bound for the lower tail of $|B_T(\emptyset,r)|$. 
\begin{proposition}[Proposition 5.3 of \cite{Addario2013local_limit_MST_complete_graphs}]\label{prop: 5.3 of Louigi}
	There exist constants $c>0$ and $C<\infty$ such that for all $r> 1$ and all $x\geq 1$, 
	\[
	\mathbb{P}\big[|B_T(\emptyset,r)|\leq\frac{r^2}{x}\big]\leq C\log r\cdot e^{-cx^{1/8}}.
	\]
\end{proposition}

\begin{proof}[Proof of  Theorem \ref{thm: improved bounds on the volume}, \eqref{eq: lower bound on the volume}]
	
	Applying Proposition \ref{prop: 5.3 of Louigi} with $x=\frac{(\log r)^{9}}{8}$ one has the following (very wasteful) estimate: there exist constants $r_0,c, C,C_1>0$ such that for all $r\geq r_0$, 
	\be\label{eq: not too few points in a ball of T}
	\mathbb{P}\Big[\big|B_T\big(\emptyset,\frac{r}{2}\big)\big|<\frac{2r^2}{(\log r)^{9}}\Big]\leq
	C \log r\cdot \exp\Big(-c\big(\frac{(\log r)^{9}}{8}\big)^{1/8}\Big)\leq \frac{C_1}{(\log r)^2} \, .
	\ee
	Similarly applying Lemma \ref{lem: lemma 4.3 of Louigi} with $x=(\log r)^{3}$ we get
	\be\label{eq: not too much points with large weights}
	\mathbb{P}\Big[ \sum_{\{i\colon X_i>1+\frac{(\log r)^{6}}{r}\}}|V(P_i)|> \frac{r^2}{(\log r)^{9}} \Big]\leq \frac{C_2}{(\log r)^2}.
	\ee
	
	Recall that for  $v\in V(T)$, if $v\in V(P_i)$ then $x(v)=X_i$. Thus  $\sum_{\{i\colon X_i>1+\frac{(\log r)^{6}}{r}\}}|V(P_i)|$ is the total number of vertices in $T$ with activation time bigger than $1+\frac{(\log r)^{6}}{r}$. Recall that $F:=\big\{ v\in V\big(B_T(\emptyset,\frac{r}{2})\big)\colon x(v)\leq 1+\frac{(\log  r)^{6}}{r} \big\}$. 
	Combining \eqref{eq: not too much points with large weights} and \eqref{eq: not too few points in a ball of T} one has that with probability at least $1-\frac{C_1+C_2}{(\log r)^2}$ the set $F$ has size at least 
	\[
	\frac{2r^2}{(\log r)^{9}}-	\frac{r^2}{(\log r)^{9}}=	\frac{r^2}{(\log r)^{9}}.
	\]
	
	As noted in item (b) of the sketch,  by Observation \ref{obser: stochastic domination of PGWA trees} and Theorem \ref{thm: dynamic constrution of M}
	on the event that $F$ has size at least $\frac{r^2}{(\log r)^{9}}$, the volume $|B_M(\emptyset,r)|$ stochastically dominates the random variable $\mathbf{1}_{A}\big(|V(T_1)|+\cdots+|V(T_m)|\big)$  from Lemma \ref{lem: paley-zygmund lower bound}. Thus, by Lemma \ref{lem: paley-zygmund lower bound}, 
	\be\label{eq: to show for the lower bound on volume growth}
	\mathbb{P}\Big[\big|B_M(\emptyset,r)\big|\leq \frac{r^3}{(\log r)^{18}}\Big]\leq \frac{C_3}{(\log r)^{2}}.
	\ee
	By the Borel--Cantelli lemma we thus have
	\[
	\liminf_{j\to\infty}\frac{\big|B_M(\emptyset,2^j)\big|}{2^{3j}/j^{18}}>0.
	\]
	Using the monotonicity of $|B_M(\emptyset,r)|$ in $r$, it follows that
	\[
	\liminf_{r\to\infty}\frac{|B_M(\emptyset,r)|}{r^3/(\log r)^{18}}>0,
	\]
	concluding the proof.
\end{proof}

\section{Effective resistance}\label{sec: resistance}
For a graph $G=(V,E)$ and functions $g,h\in\mathbb{R}^V$, we define the Dirichlet quadratic form $\mathcal{E}$ by 
\be\label{eq: quadratic form}
\mathcal{E}(g,h)=\frac{1}{2}\sum_{\substack{x,y\in V\\x\sim y}}{(g(x)-g(y))(h(x)-h(y))}
\ee
Here we regard $G$ as an electrical network with a \emph{unit resistor} on each edge of $G$ and $\mathcal{E}(g,g)$ is the energy dissipation associated with the potential $g$. Let $A,B$ be disjoint subsets of $V$. We define the effective resistance between $A$ and $B$ by 
\be\label{eq: def of effective resistance}
\reff(A,B)^{-1}:=\inf\{ \mathcal{E}(g,g)\colon g\in \mathbb{R}^V, g|_{A}=1, g|_B=0 \}.
\ee

\subsection{Lower bound on effective resistance}

Our goal is to establish the following lower bound on $\reff\big(\emptyset,B_M(\emptyset,r)^c\big)$ which is linear in $r$ up to logarithmic corrections. The linear upper bound for $\reff\big(\emptyset,B_M(\emptyset,r)^c\big)$ is straightforward, since we always have $\reff\big(\emptyset,B_M(\emptyset,r)^c\big)\leq r+1$ (Lemma 2.1 in \cite{BCK2005sub_Gaussian}). 
\begin{proposition}\label{prop: effective resistance lower bound}
	Almost surely 
	\[
	\liminf_{r\to\infty} \frac{\reff\big(\emptyset,B_M(\emptyset,r)^c\big)}{r/(\log r)^{8} }=\infty.
	\]
	Furthermore, there exists $r_0=r_0(M)\in(0,\infty)$ such that almost surely for all $r\ge r_0$ and $v\in B_M\big(\emptyset,\frac{r}{(\log r)^8}\big)$,
	\[
	\reff(v,B_M(\emptyset,r)^c)\geq  \frac{r}{(\log r)^{8}} \, .
	\]
\end{proposition}

Since $M$ is one-ended and $T$ is an infinite subtree of $M$ containing the root $\emptyset$, there is a unique infinite simple path $\gamma=(\gamma_0,\gamma_1,\ldots)$  starting from $\gamma_0=\emptyset$ in $T$. The path $\gamma$ is often called the \textit{backbone} of $T$. 
\nomenclature[Gammafuture]{$\gamma=(\gamma_0,\ldots)$ }{The unique infinite simple path in $T$ that represents the end starting from the root $\gamma_0=\emptyset$.}
Let $\scom\big(\emptyset,M\backslash(\gamma_i,\gamma_{i+1})\big)$ be the finite component of the root $\emptyset$ in $M$ after deleting the edge $(\gamma_i,\gamma_{i+1})$. Similarly let $\scom\big(\emptyset,T\backslash(\gamma_i,\gamma_{i+1})\big)$ be the finite component of the root $\emptyset$ in $T\backslash (\gamma_i,\gamma_{i+1})$. Write
\[
\mathrm{Height}\Big(\scom\big(\emptyset,M\backslash(\gamma_i,\gamma_{i+1})\big)\Big):=\max\bigg\{d(\emptyset,u)\colon u\in V\Big(\scom\big(\emptyset,M\backslash(\gamma_i,\gamma_{i+1})\big)\Big)\bigg\}.
\] 


Recall from Section \ref{sec: subsec estimates from Louigi} that $X_i$ is the weight of the $i$-th outlet $(R_i,S_{i+1})$ of $T$, where $R_{i}$ is the endpoint that is closer to the root $\emptyset$. Also recall that $I(z):=\min\{i\colon X_i\leq z\}$ for $z>1$. 
\begin{lemma}\label{lem: almost all paths to B(r) complement use a the backbone up to r over log r}
	There exists a constant $C>0$ such that for all $r\geq 2$,
	
	\be\label{eq: height of M backslash R_I not too large}
	\mathbb{P}\bigg[\mathrm{Height}\Big(\scom\big(\emptyset,M\backslash (R_{i_0},S_{i_0+1})\big)\Big)\geq r\bigg]\leq \frac{C}{(\log r)^{4/3}},
	\ee
	where $i_0=i_0(r)=I\big(1+\frac{(\log r)^{5}}{r}\big)$.
\end{lemma}

Assuming Lemma \ref{lem: almost all paths to B(r) complement use a the backbone up to r over log r} for now, it is easy to prove Proposition \ref{prop: effective resistance lower bound}. 
\begin{proof}[Proof of Proposition \ref{prop: effective resistance lower bound}]
	Note that all the outlets lie on the backbone $\gamma$ and 
	\[
	d(\emptyset,R_{i_0})\leq \mathrm{Height}\Big(\scom\big(\emptyset, M\backslash (R_{i_0},S_{i_0+1})\big)\Big).
	\]
	If $\mathrm{Height}\Big(\scom\big(\emptyset, M\backslash (R_{i_0},S_{i_0+1})\big)\Big)<r$, then  $d(\emptyset,R_{i_0})<r$ and  the vertex $R_{i_0}$ is on the path from $\emptyset$ to $\gamma_r$. Moreover each edge $(\gamma_i,\gamma_{i+1})$ is a cut-edge separating $\emptyset$ from $B_M(\emptyset,r)^c$ for $i=0,\ldots,d\big(\emptyset,  R_{i_0}\big)$. Therefore if $\mathrm{Height}\Big(\scom\big(\emptyset, M\backslash (R_{i_0},S_{i_0+1})\big)\Big)<r$, then by Nash-Williams inequality (for instance see Section 2.5 in \cite{LP2016}) one immediately has that \[
	\reff\big(\emptyset,B_M(\emptyset,r)^c\big)\geq d\big(\emptyset,  R_{i_0}\big).
	\]
	
	By Lemma \ref{lem: lemma 4.4 of Louigi} 
	\be\label{eq: R_I not too close to the root}
	\mathbb{P}\Big[d\big(\emptyset,  R_{i_0}\big) \leq \frac{r}{(\log r)^{7}}\Big]\leq \frac{C}{(\log r)^{4/3}}.
	\ee 
	Hence by \eqref{eq: R_I not too close to the root} and \eqref{eq: height of M backslash R_I not too large} 
	\[
	\mathbb{P}\big[ \reff\big(\emptyset,B_M(\emptyset,r)^c\big)\leq \frac{r}{(\log r)^{7}}\big]\leq \frac{C}{(\log r)^{4/3}}.
	\] By Borel--Cantelli lemma one has that almost surely
	\[
	\liminf_{j\to\infty}\frac{\reff\big(\emptyset,B_M(\emptyset,2^j)^c\big)}{2^j/j^{7}}>0. 
	\]
	Since  $\reff\big(\emptyset,B_M(\emptyset,r)^c\big)$ is increasing in $r$ it follows that almost surely
	\[
	\liminf_{r\to\infty} \frac{\reff\big(\emptyset,B_M(\emptyset,r)^c\big)}{r/(\log r)^{7} }>0 \, ,
	\]
	so almost surely
	\[
	\liminf_{r\to\infty} \frac{\reff\big(\emptyset,B_M(\emptyset,r)^c\big)}{r/(\log r)^{8} }=\infty.
	\] 
	It also follows that there exists a random variable $r_0=r_0(M)\in(0,\infty)$ such that almost surely for all $r\ge r_0$, 
	\[
	\reff\big(\emptyset,B_M(\emptyset,r)^c\big)\geq \frac{2r}{(\log r)^8}\,.
	\]
	For  any $v\in B_M\big(\emptyset,\frac{r}{(\log r)^8}\big)$, we have that $\reff(\emptyset,v)\leq d(\emptyset,v)\leq \frac{r}{(\log r)^8}$.
	Thus by the triangle inequality for effective resistance (for instance see \cite[Exercise 2.67]{LP2016}) we have that for any $v\in B_M\big(\emptyset,\frac{r}{(\log r)^8}\big)$,
	\[
	\reff(v,B_M(\emptyset,r)^c)\geq \reff\big(\emptyset,B_M(\emptyset,r)^c\big)-\reff(\emptyset,v) \ge \frac{r}{(\log r)^{8}} \, .\qedhere
	\]
\end{proof}

It remains to prove Lemma \ref{lem: almost all paths to B(r) complement use a the backbone up to r over log r}.
We first sketch the ideas of the proof. 
\begin{enumerate}

	\item[(a)] \label{item: a}
	We first show that 
	\be\label{eq: step a in lower bound of resistance}
	\mathbb{P}\bigg[\mathrm{Height}\Big(\scom\big(\emptyset,T\backslash (R_{i_0},S_{i_0+1})\big)\Big)\leq \frac{r}{2}\bigg]\geq 1-\frac{C}{(\log r)^{4/3}}. 
	\ee

	\item[(b)] Next we use Lemma \ref{lem: lemma 4.3 of Louigi} and \ref{lem: lemma 4.4 of Louigi}  to show that
	\be\label{eq: step b in lower bound of resistance}
	\mathbb{P}\bigg[ X_{i_0}>1+\frac{(\log r)^{3}}{r} \textnormal{ and } \sum_{i=1}^{i_0}|V(P_i)|\leq r^2 \bigg]\geq 1-\frac{C}{(\log r)^{4/3}}.
	\ee
	
	\item[(c)] Conditioned on the events in \eqref{eq: step a in lower bound of resistance} and \eqref{eq: step b in lower bound of resistance}, by Theorem \ref{thm: dynamic constrution of M}, Observation \ref{obser: stochastic domination of PGWA trees}, Proposition \ref{prop: height of M_v} and a union bound  one has that with high probability all the trees $M_v$ have height less than $\frac{r}{2}$, where $v$ runs over $\bigcup_{i=1}^{i_0} V(P_i)$. In particular this implies Lemma \ref{lem: almost all paths to B(r) complement use a the backbone up to r over log r}.
\end{enumerate}

We begin with a lemma needed for the proof of \eqref{eq: step a in lower bound of resistance}. 
Recall that the conditional density of $X_{n+1}$ given that $X_n=x$ is  $\frac{\theta'(y)}{\theta(x)}$ for $y\in(1,x)$ (see \eqref{eq: transition kernel for X_n}), where $\theta(x)=\mathbb{P}\big(|\mathrm{PGW}(x)|=\infty\big)$. 
Since $\theta'(y)\to 2$ as $y\downarrow 1$, for large $n$ the ratios $\frac{X_{n+1}-1}{X_n-1}$ are approximately distributed as $\mathrm{Uniform}[0,1]$ random variables and thus the difference $X_n-1$ typically decreases by a factor of two after increasing $n$ by a constant number.
Thus $I(1+\frac{1}{r})$ is typically of order $\log r$,  and the following  estimate is straightforward to obtain (recall that $I(z):=\min\{i\colon X_i\leq z\}$). 
\begin{lemma}\label{lem: decay of I}
	There exist constants $c>0$ and $C<\infty$ such that for all $r\geq 1$,
	\[
	\mathbb{P}\big[ I(1+\frac{1}{r}) \geq (\log r)^2\big]\leq Ce^{-c(\log r)^2}.
	\]
\end{lemma}
\begin{proof}
	The proof is similar to Lemma 4.2 in \cite{Addario2013local_limit_MST_complete_graphs}, and we provide the details for completeness. As noted in the proof of Lemma 4.2 in \cite{Addario2013local_limit_MST_complete_graphs}, for $i\geq 1$,
	\begin{eqnarray*}
		\mathbb{P}\big(I(2)>i\big)&\leq& \prod_{j=1}^{i-1}\sup_{y>1}\mathbb{P}\big[X_{j+1}>2\mid X_j=1+y\big]\nonumber\\
		&=& \prod_{j=1}^{i-1}\sup_{y>1}\int_{2}^{1+y}\frac{\theta'(s)}{\theta(1+y)}~\din s
		= \prod_{j=1}^{i-1}\sup_{y>1}\frac{\theta(1+y)-\theta(2)}{\theta(1+y)}\nonumber\\
		&\leq&(1-\theta(2))^{i-1}. 
	\end{eqnarray*}
	
	For $y>0$, by Lagrange's mean value theorem there exist $\xi_1\in(1+y/2,1+y)$ and $\xi_2\in(1,1+y/2)$ such that $\theta(1+y)-\theta(1+y/2)=\theta'(\xi_1)\cdot \frac{y}{2}$ and $\theta(1+y/2)=\theta'(\xi_2)\cdot \frac{y}{2}$ (since $\theta(1)=0$). Since $\theta$ is concave, $\theta'(\xi_1)\leq \theta'(\xi_2)$. Hence one has 
	that 
	\begin{eqnarray}\label{eq: dominated by Bernoulli n 1 over 3}
		\mathbb{P}\big[X_{j+1}-1\geq y/2 \mid X_j-1=y\big]&=&\frac{\theta(1+y)-\theta(1+y/2)}{\theta(1+y)}\nonumber\\
		&=&\frac{\theta(1+y)-\theta(1+y/2)}{\theta(1+y)-\theta(1+y/2)+\theta(1+y/2)}\nonumber\\
		&=&\frac{\theta'(\xi_1)\cdot \frac{y}{2}}{\theta'(\xi_1)\cdot \frac{y}{2}+\theta'(\xi_2)\cdot  \frac{y}{2}}\leq \frac{1}{2}.
	\end{eqnarray}
	By \eqref{eq: dominated by Bernoulli n 1 over 3} the number of $j\in\big[\frac{(\log r)^2}{2},(\log r)^2-2\big]$ such that $X_{j+1}-1\leq \frac{X_j-1}{2}$  dominates  a Binomial$\big(\frac{(\log r)^2}{2}-1,\frac{1}{2}\big)$ random variable. 
	
	If $I(2)\leq \frac{(\log r)^2}{2}$, then $X_{\frac{(\log r)^2}{2}}\leq 2$ by the definition of $I$. 
	If in addition the number of $j\in\big[\frac{(\log r)^2}{2},(\log r)^2-2\big]$ such that $X_{j+1}-1\leq \frac{X_j-1}{2}$ is more than $\log_2(r)$, then 
	\[
	X_{(\log r)^2-1}-1\leq \big(X_{\frac{(\log r)^2}{2}}-1\big)\cdot \big(\frac{1}{2}\big)^{\log_2(r)}
	\leq (2-1)\cdot \frac{1}{r}=\frac{1}{r}
	\]
	and then $I(1+\frac{1}{r})\leq(\log r)^2-1<(\log r)^2$.
	Therefore  
	\begin{eqnarray*}
		\mathbb{P}\big[ I(1+\frac{1}{r}) \geq (\log r)^2\big]&\leq&	\mathbb{P}\big(I(2)>\frac{(\log r)^2}{2}\big)+
		\mathbb{P}\Big[ \mathrm{Bin}\big(\frac{(\log r)^2}{2}-1,\frac{1}{2}\big)\leq \log_2(r) \Big]\\
		&\leq& \big(1-\theta(2)\big)^{\frac{(\log r)^2}{2}-1}+
		\mathbb{P}\Big[ \mathrm{Bin}\big(\frac{(\log r)^2}{2}-1,\frac{1}{2}\big)\leq \log_2(r) \Big]\\
		&\leq& Ce^{-c(\log r)^2},
	\end{eqnarray*}
	where the Chernoff--Cram\'{e}r theorem on large deviations (for instance see  \cite{Durrett_book_5th_edition}, (2.7.2) and Lemma 2.7.2) is used in the last step.
\end{proof}

Another ingredient for the proof of \eqref{eq: step a in lower bound of resistance} is the following assertion, giving a tail bound on the sum of the diameters of ponds in $T$.
\begin{proposition}[Proposition 5.2 of \cite{Addario2013local_limit_MST_complete_graphs}]
	\label{prop: prop 5.2 of Louigi}
	There exist constants $c>0$ and $C<\infty$ such that for all $x\geq 1$ and all $r> 1$, 
	\[
	\mathbb{P}\Big[\sum_{\{i\colon X_i>1+\frac{1}{r}\} } \mathrm{diam}(P_i)\geq xr \Big]\leq C\log r\cdot e^{-c\sqrt{x}}.
	\]
\end{proposition}

\begin{proof}[Proof of Lemma \ref{lem: almost all paths to B(r) complement use a the backbone up to r over log r}]
	We follow the sketch given earlier. Recall that $i_0=i_0(r)=I\big(1+\frac{(\log r)^{5}}{r}\big)$.
	By Lemma \ref{lem: lemma 4.4 of Louigi}, 
	\be\label{eq: X_I not too close to 1}
	\mathbb{P}\Big[ X_{i_0} \leq 1+\frac{(\log r)^{3}}{r}\Big]\leq \frac{C}{(\log r)^{4/3}}. 
	\ee
	If 	$ X_{i_0} > 1+\frac{(\log r)^{3}}{r}$, then using the fact that $X_i$ is decreasing in $i$ one has that
	\[
	\sum_{i=1}^{ i_0 } \mathrm{diam}(P_i)\leq \sum_{\{i\colon X_i>1+\frac{(\log r)^{3}}{r} \}} \mathrm{diam}(P_i). 
	\]
	By Proposition \ref{prop: prop 5.2 of Louigi} applied with $x=(\log r)^{3}/4$ we get 
	\[
	\mathbb{P}\Big[\sum_{\{i\colon X_i>1+\frac{(\log r)^{3}}{r} \} } \mathrm{diam}(P_i)\geq \frac{r}{4} \Big]\leq C\big( \log \frac{r}{(\log r)^{3}}\big)\cdot e^{-c\sqrt{(\log r)^{3}/4}} \leq \frac{C_1}{(\log r)^{4/3}}. 
	\]	
	Hence 
	\[
	\mathbb{P}\bigg[\sum_{i=1}^{ i_0 } \mathrm{diam}(P_i)\geq \frac{r}{4} \bigg]\leq \frac{C}{(\log r)^{4/3}}. 
	\]	
	By Lemma \ref{lem: decay of I},  
	\[
	\mathbb{P}\Big[  i_0 \geq \frac{r}{4} \Big]
	\leq \frac{C}{(\log r)^{4/3}}. 
	\]
	Since the invasion tree $T$ can be viewed as the concatenation of ponds via the outlets, we get
	\[
	\mathrm{Height}\Big(\scom\big(\emptyset,T\backslash (R_{i_0},S_{i_0+1})\big)\Big)\leq
	i_0+\sum_{i=1}^{ i_0} \mathrm{diam}(P_i) \, ,
	\]
	therefore 
	\be\label{eq: height of component in T deleting an edge is not too large}
	\mathbb{P}\Big[ 	\mathrm{Height}\Big(\scom\big(\emptyset,T\backslash (R_{i_0},S_{i_0+1})\big)\Big) \geq \frac{r}{2}  \Big]\leq \frac{C}{(\log r)^{4/3}}. 
	\ee

	By \eqref{eq: X_I not too close to 1} and the fact that $X_i$ is decreasing in $i$ one has that $ X_{i} > 1+\frac{(\log r)^{3}}{r}$ for all $i\in\{1,\ldots, i_0\}$ with probability at least $1-\frac{C_1}{(\log r)^{4/3}}$. 
	
	We now handle the volume of the ponds similarly. If $X_{i_0} > 1+\frac{(\log r)^{3}}{r}$, then as before
	\[
	\sum_{i=1}^{ i_0 } |V(P_i)|\leq \sum_{\{i\colon X_i>1+\frac{(\log r)^{3}}{r}\} } |V(P_i)|. 
	\]
	By Lemma \ref{lem: lemma 4.3 of Louigi}, 
	\[
	\mathbb{P}\Big[ \sum_{\{i\colon X_i>1+\frac{(\log r)^{3}}{r}  \}} |V(P_i)|\geq r^2 \Big]\leq C\big(\log \frac{r}{ (\log r)^{3} } \big)e^{-c\sqrt{ (\log r)^{6} }}\leq \frac{1}{(\log r)^{4/3}}.
	\]
	By this and \eqref{eq: X_I not too close to 1} we get 
	\be\label{eq: sum of size the first izero ponds are not too large}
	\mathbb{P}\bigg[\sum_{i=1}^{ i_0 } |V(P_i)|\geq r^2 \bigg]\leq \frac{C}{(\log r)^{4/3}}. 
	\ee
	
	Recall that for each $v\in V(P_i)$ one has that $x(v)=X_i$. If 	
	$X_{i_0} > 1+\frac{(\log r)^{3}}{r}$, then one has that $x(v)>1+\frac{(\log r)^{3}}{r}$ for all $v\in \bigcup_{i=1}^{i_0}V(P_i)$. Conditioned on $X_{i_0} > 1+\frac{(\log r)^{3}}{r}$, by Theorem \ref{thm: dynamic constrution of M} and Observation \ref{obser: stochastic domination of PGWA trees}, for all $v\in \bigcup_{i=1}^{i_0}V(P_i)$ the corresponding tree $M_v$ is stochastically dominated by a random tree with law $\mathbb{P}_{1+\frac{(\log r)^{3}}{r}}$. Hence by Proposition \ref{prop: height of M_v} we have 
	\[
	\mathbb{P}\Big[ \mathrm{Height}(M_v)\geq \frac{r}{2}\,\big|\, T,X_{i_0} > 1+\frac{(\log r)^{3}}{r}, v\in \bigcup_{i=1}^{i_0}V(P_i)\Big]
	\leq \frac{C (\log r)^9}{r^3} \, .
	\]

	Now we are ready to obtain the desired conclusion. If 
	\[
	\mathrm{Height}\Big(\scom\big(\emptyset,T\backslash (R_{i_0},S_{i_0+1})\big)\Big)\leq \frac{r}{2} \textnormal{ and }\mathrm{Height}\Big(\scom\big(\emptyset,M\backslash(R_{i_0},S_{i_0+1})\big)\Big)\geq r,
	\] then there exists a vertex $v\in \bigcup_{i=1}^{i_0}V(P_i)$ with $\mathrm{Height}(M_v)\geq \frac{r}{2}$. If furthermore one has that $X_{i_0} > 1+\frac{(\log r)^{3}}{r}$ and $\sum_{i=1}^{ i_0 } |V(P_i)|\leq r^2$, then by a union bound the probability that at least one of these $M_v$'s have height at least $r/2$ is at most $r^2\cdot\frac{C (\log r)^9}{r^3} \ll \frac{1}{(\log r)^{4/3}}$. Hence 
	\begin{eqnarray*}
		\mathbb{P}\Big[\mathrm{Height}\Big(\scom\big(\emptyset,M\backslash (R_{i_0},S_{i_0+1})\big)\Big)\geq r\Big]
		&\leq&		\mathbb{P}\Big[ X_{i_0} \leq 1+\frac{(\log r)^{3}}{r}\Big]+	\mathbb{P}\bigg[\sum_{i=1}^{ i_0 } |V(P_i)|\geq r^2 \bigg]\nonumber\\
		&&+\mathbb{P}\Big[ 	\mathrm{Height}\Big(\scom\big(\emptyset,T\backslash (R_{i_0},S_{i_0+1})\big)\Big) \geq \frac{r}{2}  \Big]\nonumber\\
		&&+\frac{1}{(\log r)^{4/3}}
		\leq  \frac{C}{(\log r)^{4/3}},
	\end{eqnarray*}
	where in the last inequality we used the inequalities \eqref{eq: X_I not too close to 1}, \eqref{eq: height of component in T deleting an edge is not too large} and \eqref{eq: sum of size the first izero ponds are not too large}.\end{proof}

\section{Spectral and diffusive properties of $M$}\label{sec: spectral and diffusive properties}

In most random fractals, knowing the growth of the volume and effective resistance allows one to sharply estimate certain random walk asymptotics such as the spectral dimension and typical displacement. These relationships were established in a series of studies \cite{BCK2005sub_Gaussian,BJKS2008RW_IIC_oriented_percolation,BarlowKumagaiTrees06} and were further generalized and synthesized by Kumagai and Misumi in \cite{Kumagai_Misumi2008}. In this section we only apply results from \cite{Kumagai_Misumi2008}. 

Since we do not have sharp enough tail estimates on the volume growth and effective resistance, we cannot use the main theorem from \cite{Kumagai_Misumi2008} directly. Nevertheless, we can derive the following \cref{thm: KM2008 deterministic setting} for deterministic graphs based on \cite{Kumagai_Misumi2008}. This theorem will then be employed to prove both \cref{thm: spectral dimension} and \cref{thm: diffusive property}.

Suppose $G$ is an infinite, locally finite, connected graph with a marked vertex denoted by $0$, and let $d=d_G$ be the graph distance on $G$. Write $B(R)=B_G(0,R)$ for the ball with radius $R$ centered at $0$ and let 
\[
V(R):=\sum_{x\in B(R)}\deg(x). 
\]
Let $(Y_n)$ be a (discrete-time) simple random walk on $G$ and recall the transition density $p_n(x,y)$ is defined as 
\[
p_n(x,y):=P^x(Y_n=y)/\deg(y). 
\]

The following theorem collects several estimates proved in \cite{Kumagai_Misumi2008}; we provide a brief proof below.

\begin{theorem}\label{thm: KM2008 deterministic setting}
	Suppose $G$ is a deterministic, infinite, locally finite, connected graph with a marked vertex $0\in G$ and $d(\cdot,\cdot)$ is the graph distance on $G$. If there exist constants $\eta\ge 0$, $D\ge 1$, $C\geq 1$ and $\alpha\in(0,1]$ such that for all $R\geq 2$ 
	\be\label{eq:volume growth assumption in adaption of KM2008}
	C^{-1} (\log R)^{-\eta} R^D \le  V(R) \leq C (\log R)^{\eta}R^D, 
	\ee
	\be\label{eq:two-point effective resistance assumption in KM2008}
	\reff(0,y)\leq C (\log R)^{\eta}d(0,y)^{\alpha},\,\forall \,y\in B(R),
	\ee 
	and 
	\be\label{eq:effective resistance assumption in adaption of KM2008}
	\reff(x,B(R)^c) \geq C^{-1} (\log R)^{-\eta}R^{\alpha},\quad \forall\, x \in B\Big(0,\frac{R}{C(\log R)^{\eta}}\Big) \, ,
	\ee
	then there exist constants $\beta=\beta(\eta,D,\alpha) \in(0,\infty)$ and $C_1\geq 1$ such that the following statements hold\footnote{As mentioned in Remark~\ref{rem: other random walk asymptotics}, other random walk asymptotics such as exit time of balls and the range of the walk can also be derived and we omit the details.}:
	\begin{enumerate}
		\item[(a)] The transition density satisfies 
		\[
		C_1^{-1} (\log n)^{-\beta} n^{-\frac{D}{D+\alpha}}	\leq p_{2n}(0,0)\leq C_1(\log n)^{\beta} n^{-\frac{D}{D+\alpha}} \, ,
		\]
		for any $n\geq 2$. In particular $d_s(G)=\frac{2D}{D+\alpha}$. 
		\item[(b)] For any $n\ge 2$,
		\[
		P^0\Big[ C_1^{-1}(\log n)^{-\beta}n^{\frac{1}{D+\alpha}} \leq d(0,Y_n) \leq C_1(\log n)^{\beta}n^{\frac{1}{D+\alpha}}  \Big] \geq 1-\frac{C_1}{\log n}.
		\]
		
		%
		%
	\end{enumerate}
\end{theorem}

Before proving \cref{thm: KM2008 deterministic setting}, let us first derive \cref{thm: spectral dimension} and \cref{thm: diffusive property} from it. 

\begin{proof}[Proof of \cref{thm: spectral dimension} and \cref{thm: diffusive property}]
	Since $M$ is a tree,
	\[
	|B_M(\emptyset,r)|\leq V(\emptyset,r)=\sum_{x\in B_M(\emptyset,r)}\deg(x)\leq 2|B_{M}(\emptyset,r+1)|.
	\] 
	By \cref{thm: improved bounds on the volume} we can conclude that, almost surely, $(M,\emptyset)$ satisfies condition \eqref{eq:volume growth assumption in adaption of KM2008} with $D=3$ and $\eta=18$. Since $\reff(\emptyset,y)\le d(\emptyset,y)$,  condition \eqref{eq:two-point effective resistance assumption in KM2008} holds with $\alpha=1$ and $\eta=0$. Moreover, by Proposition~\ref{prop: effective resistance lower bound}, condition \eqref{eq:effective resistance assumption in adaption of KM2008} holds with $\alpha=1$ and $\eta=8$. As a result, \cref{thm: spectral dimension} and \cref{thm: diffusive property} can be directly inferred from \cref{thm: KM2008 deterministic setting}.
\end{proof}


\begin{proof}[Proof of \cref{thm: KM2008 deterministic setting}] The proof is an application of various estimates in \cite{Kumagai_Misumi2008}. The proof of the upper bound on $p_{2n}(0,0)$ is an immediate application of Proposition 3.1(a) of \cite{Kumagai_Misumi2008} with $R=\lfloor (2n)^{\frac{1}{D+\alpha}}\rfloor$; note that in our setting $v(R)=R^D$, $r(R)=R^{\alpha}$ and $\lambda=C(\log R)^{\eta}$, moreover, $\mathcal{I}(n)$ of \cite{Kumagai_Misumi2008} simply equals $n^{\frac{1}{D+\alpha}}$. The assumptions of Proposition 3.1(a) of \cite{Kumagai_Misumi2008} (which are (3.1) of that paper) are contained in the assumptions of \cref{thm: KM2008 deterministic setting} and so the upper bound on $p_{2n}(0,0)$ follows.
	
	To prove the lower bound on $p_{2n}(0,0)$ we apply Proposition 3.2(b) of \cite{Kumagai_Misumi2008}. There are more parameters to set and conditions to verify than previously. First note that in our setting we have that $C_2$ and $\alpha_1$ defined in (1.12) of \cite{Kumagai_Misumi2008} can be taken to be $1$ and $\alpha$ respectively. Next we take $\lambda=m=C(\log R)^{\eta}$ and $\varepsilon = (4m\lambda)^{-1/\alpha}$ so that in particular $\varepsilon \leq {1 \over C(\log R)^\eta}$. These choices guarantee that (3.3) of \cite{Kumagai_Misumi2008} holds as well as $\reff(0,y)\leq \lambda r(d(0,y))$ for all $y\in B(R)$ again due to the assumptions of our theorem. Then (3.6) of \cite{Kumagai_Misumi2008} states that there exists $c_1>0$ such that for $n\le \frac{R^\alpha V'}{8m}$, 
	\be\label{eq: transition density lower bound before rescaling}
	p_{2n}(0,0)\geq \frac{c_1(V')^2}{m^2\lambda^2V^3}, 
	\ee
	where $V'=V\big(\frac{1}{2}\varepsilon R\big)$ and $V=V(R)$. 
	Since 
	\be\label{eq:V'}
	V'\stackrel{\eqref{eq:volume growth assumption in adaption of KM2008}}{\ge} C^{-1}\big(\log (\frac{\varepsilon R}{2})\big)^{-\eta}\big(\frac{1}{2}\varepsilon R\big)^{D} \ge c_{2}(\log R)^{-\eta-\frac{2\eta D}{\alpha}}R^D 
	\ee
	and 
	\be\label{eq:V}
	V\stackrel{\eqref{eq:volume growth assumption in adaption of KM2008}}{\le}C(\log R)^{\eta}R^D,
	\ee
	we have for $n =\lfloor \frac{c_2}{8C}(\log R)^{-2\eta-\frac{2\eta D}{\alpha}} R^{D+\alpha}\rfloor \stackrel{\eqref{eq:V'}}{\le}\frac{R^\alpha V'}{8m}$, 
	\[
	p_{2n}(0,0)\ge \frac{c_1(V')^2}{m^2\lambda^2V^3}\stackrel{\eqref{eq:V'},\eqref{eq:V}}{\ge} 
	c_3 (\log R)^{-9\eta-\frac{4\eta D}{\alpha}} R^{-D}.
	\]
	Set $R$ so that $n=\lfloor \frac{c_2}{8C}(\log R)^{-2\eta-\frac{2\eta D}{\alpha}} R^{D+\alpha}\rfloor $ and get the desired lower bound on $p_{2n}(0,0)$.

	Finally for the displacement  bounds, we will apply Proposition~3.5 (b) and (c) of \cite{Kumagai_Misumi2008}. Again in our setting $v(R)=R^D$ and $r(R)=R^{\alpha}$ and the exponent $d_1$ defined in (1.12) of \cite{Kumagai_Misumi2008} can be taken to be $D$. Moreover, for $\lambda>1$, the set  $J(\lambda)$ in  Definition 1.1 of \cite{Kumagai_Misumi2008} becomes
	\begin{align*}
		J(\lambda)=&\Big\{ R\ge 1\colon \lambda^{-1}R^D\le V(R)\leq \lambda R^D, \reff\big(0,B(R)^c\big)\ge \lambda^{-1}R^{\alpha},\\ 
		&\,\,\,\,\, \reff(0,y)\le \lambda d(0,y)^{\alpha},\forall\, y\in B(R)\Big\}.
	\end{align*}
	We take $\lambda=C(\log R)^{\eta}$ as before. The conditions in \cref{thm: KM2008 deterministic setting} guarantee that $R\in J(\lambda)$, and $cR\in J(\lambda)$ for any $c\in(2/R,1)$.
	Note that the constants $c_i(\lambda)$'s in \cite[Proposition 3.5]{Kumagai_Misumi2008} have the form $c_i(\lambda)=c_i\lambda^{\pm q_i}$ for some positive universal constants $c_i$ and $q_i$ (see \cite[first paragraph in page~925, above Proposition 3.3]{Kumagai_Misumi2008}). Hence if we set the parameter $M$ in \cite[Proposition~3.5 (b)]{Kumagai_Misumi2008} to be $M=(\log n)^{\beta}$ for some  $\beta>0$, then for $R=M\mathcal{I}(n)=(\log n)^{\beta}n^{\frac{1}{D+\alpha}}$ and $\lambda=C(\log R)^{\eta}$, we can choose $\beta$ sufficiently large\footnote{In this paragraph we assume $n\ge3$ so that $\log n>1$.} so that   $R,c_5(\lambda)R/M,c_6(\lambda)R/M\in J(\lambda)$ and $M^{\alpha}\ge c_7(\lambda) \log n$. Therefore the upper bound on the displacement follows directly from Proposition~3.5 (b)  of \cite{Kumagai_Misumi2008}. Similarly for $R=n^{\frac{1}{D+\alpha}}$ and $\lambda=C(\log R)^{\eta}$ if we set the parameter $\theta=(\log n)^{-\beta}$ for some sufficiently large $\beta$ so that $c_8(\lambda)\theta^D\leq \frac{C_1}{\log n}$, then the lower bound on the displacement follows directly from Proposition~3.5 (c)  of \cite{Kumagai_Misumi2008}.
\end{proof}

\section{A local limit theorem for minimal spanning trees}\label{sec: local limit thm for MSTs}

Our goal in this section is to prove the following local limit theorem for minimal spanning trees. 
An analogous theorem about uniform spanning trees, with a different limiting object, can be found in  \cite[Theorem 1.1]{Nachmias_Peres2021local_limit}. 
\begin{theorem}\label{thm: unweight local limit for regular graph sequence}
	Let $\{G_n\}$ be a sequence of finite, simple, connected, regular graphs with degree tending to infinity and let $\{W_n(e)\}_{e\in G_n}$ be i.i.d~uniform$[0,1]$ weights. Let $M_n$ be the minimal spanning tree of $(G_n,W_n)$ viewed as an unweighted graph. Then the local limit of $M_n$ is the component of the root $\emptyset$ in the WMSF on the PWIT. 
\end{theorem}

Theorem \ref{thm: unweight local limit for regular graph sequence} is about \emph{unweighted} graphs and we will deduce it via a weighted version due to Aldous and Steele \cite{Aldous_Steele2004objective_method}. The proof shows under which conditions one can deduce unweighted convergence from weighted convergence. This  does not hold always, and certain natural conditions are necessary and sufficient, see \cref{thm:WeightedUnweighted}.  


\subsection{Rooted weighted/unweighted graphs, their topology and local limits}

The topology of rooted graphs and local limits was introduced by Benjamini and Schramm \cite{Benjamini_Schramm2001local_limit}.
Given a locally finite graph $G=(V,E)$ and a vertex $\rho\in V$, the pair $(G,\rho)$ is a {\bf rooted graph}. Two rooted graphs $(G_1,\rho_1)$ and $(G_2,\rho_2)$ are equivalent if there is a graph isomorphism $\phi:V(G_1)\to V(G_2)$ which also preserves the root. Let $\Gb$ be the set of equivalent classes of locally finite rooted graphs viewed up to root-preserving graph isomorphisms. When discussing a rooted graph $(G,\rho)$ we always mean its equivalence class in $\Gb$. 

Next we define a metric on $\Gb$. The distance between two rooted graphs $(G_1,\rho_1)$ and $(G_2,\rho_2)$ is given by 
\[
\dist\big((G_1,\rho_1),(G_2,\rho_2)\big):=\frac{1}{1+\alpha},
\]
where $\alpha$ is the supremum of those integer $r>0$ such that there is a root-preserving isomorphism between the two rooted balls $\big(B_{G_1}(\rho_1,r),\rho_1\big)$ and $\big(B_{G_2}(\rho_2,r,\rho_2\big)$; and $\alpha=\infty$ if the two rooted graphs are equivalent. This metric induces a Borel $\sigma$-algebra on $\Gb$ and makes it a Polish space (see \cite{AL2007}). Hence one can talk about convergence in distribution for Borel probability measures on $\Gb$.

\begin{definition}[Local limit]\label{def: local limit}
	A random element $(G,\rho)$ of $\Gb$ is called the \emph{local limit} of a sequence of finite (possibly random) graphs $(G_n)_{n\geq1}$ if for any integer $r>0$, the random rooted graphs $\big(B_{G_n}(\rho_n,r),\rho_n\big)$ converge in distribution to $\big(B_G(\rho,r),r\big)$, where   $\rho_n$ is a uniformly chosen vertex from $V(G_n)$ that is chosen independently of each other conditioned on $(G_n)_{n\geq1}$. We denote this by
	$$ G_n \stackrel[n \to \infty]{d}{\longrightarrow} (G,\rho) \, .$$
\end{definition}




We now turn to discuss weighted graphs. Given a connected graph $G=(V,E)$ and a weight function $l:E\to(0,\infty)$ on its edges, we let $d_W:V\times V\to (0,\infty)$ be the weighted distance, namely, 
\[
d_W(u,v):=\inf\Big\{\sum_{e\in \gamma}l(e)\colon \gamma \textnormal{ is a path between }u \textnormal{ and }v \Big\}.
\] 
\nomenclature[dweight]{$d_W(u,v)$}{The weighted graph distance on a weighted graph}
We also write 	\[
B_G^W(v,r):=\{ x\colon d_W(v,x)\leq r \}
\]
for a ball of radius $r$, centered at $v$ with respect to this weighted distance $d_W$. We often view $B_G^W(v,r)$ as the induced subgraph in $G$ of the vertex set $\{ x\colon d_W(v,x)\leq r \}$. 
\begin{definition}[Definition 2.1 in \cite{Aldous_Steele2004objective_method}; Geometric graphs]
	\label{def: geometric graphs}
	If $G=(V,E)$ is a connected, undirected graph with a finite or countable infinite vertex set $V$ and if $\ell:E\to(0,\infty)$ is an edge weight function that makes $G$ \emph{locally finite} in the sense that for each vertex $v$ and each real $t>0$, the ball $B_G^W(v,t)$
	is finite, then $(G,\ell)$ is called a \emph{geometric graph}. When there is also a distinguished vertex $v$, we say that $(G,v,\ell)$ is a \emph{rooted geometric graph} with \emph{root} $v$. Two rooted geometric graphs are equivalent if there is a root-preserving isomorphism between the two rooted graphs which also preserves the edge weights.  Let  $\Gs$ be the set of equivalence classes of rooted geometric graphs. 
	Lastly, for $(G,\rho,\ell)\in\Gs$, we say that $t>0$ is a \emph{continuity point} of $(G,\rho,\ell)$ if no vertex of $G$ is exactly at a weighted distance $t$ from the root $\rho$. 
	
\end{definition}

\begin{example}
	Although the vertex degrees of the Ulam--Harris tree $U$ defined in Section \ref{sec: intro: MSF on PWIT} are all infinite, the PWIT $(U,\emptyset,W)$ is a.s.~a geometric graph in the sense of Definition \ref{def: geometric graphs}. 
\end{example}

Next we define a topology on $\Gs$. Suppose $(G_n,\rho_n,\ell_n)_{n\geq1}$ and $(G_\infty,\rho,\ell)$ are elements in $\Gs$.
We say that $(G_n,\rho_n,\ell_n)$ converges to $(G_\infty,\rho,\ell)$ if for each continuity point $t$ of $(G_\infty,\rho,\ell)$, there is $n_0$ such that for all $n\geq n_0$ there exists a root-preserving isomorphism $\phi_n$ from $\big(B_{G_\infty}^W(\rho,t),\rho\big)$ to $\big(B_{G_n}^W(\rho_n,t),\rho_n\big)$  such that  $\ell_n\big(\phi_n(e)\big)\to \ell(e)$ for each edge $e$ of $B_{G_\infty}^W(\rho,t)$ as $n\to\infty$. It can easily be seen that this topology is metrizable and in fact makes $\Gs$ a Polish space. Hence one can talk about convergence in distribution of Borel probability measures on $\Gs$ (this convergence was referred as \emph{local weak convergence} in \cite{Aldous_Steele2004objective_method}).

\begin{definition}[local limit of weighted graphs]\label{def: local weighted limit}
	A random element $(G,\rho, \ell)$ of $\Gs$ is called the \emph{local limit} of a sequence of finite (possibly random) geometric graphs $(G_n, \ell_n)_{n\geq1}$ if the random rooted geometric graphs $(G_n, \rho_n, \ell_n)_{n\geq1}$ converge in distribution to $(G, \rho, \ell)$ where $\rho_n$ is a uniformly chosen vertex from $V(G_n)$ that is chosen independently of each other conditioned on $(G_n)_{n\geq1}$. We denote this by
	$$ (G_n, \ell_n) \stackrel[n \to \infty]{d}{\longrightarrow} (G,\rho,\ell) \, .$$
\end{definition}


\begin{remark}Note that the local limit defined earlier on $\Gb$ can be viewed as a special case of the local limit on $\Gs$ if we view the unweighted graphs as  weighted graphs with weights $1$ on all the edges and pick the roots uniformly at random. Due to this and in order to avoid further terminology we refer to the convergence defined in \cref{def: local limit} and \cref{def: local weighted limit} as a \emph{local limit} but it will always be clear from the context (i.e., whether weights are present or not) which topology we work with. 
\end{remark}

\subsection{The mass-transport principle}
Similarly to $\Gs$, one can define the space $\Gss$ as equivalent classes of weighted graphs with an ordered pair of distinguished vertices. One can define a natural topology on $\Gss$ similar to $\Gs$. An useful technique for the proof of Theorem \ref{thm: unweight local limit for regular graph sequence} is the mass-transport principle. 

\begin{definition}[Definition 2.1 in \cite{AL2007} adapted to $\Gs$]
	Let $\mu$ be a probability measure on $\Gs$. We call $\mu$ \emph{unimodular} if it obeys the \emph{mass-transport principle}: For any Borel measurable function $f:\Gss\to[0,\infty]$, we have 
	\be\label{eq: MTP}
	\int \sum_{x\in V(G)} f(G,o,x)~\din \mu \big((G,o)\big)=\int \sum_{x\in V(G)} f(G,x,o)~\din \mu \big((G,o)\big).
	\ee
	If we view $f(G,o,x)$ as the amount of mass sent out from $o$ to $x$, then the equality \eqref{eq: MTP} is usually interpreted as the expected mass sent out by the root equals the expected mass received. 
\end{definition}

It is well known (and not difficult to prove) that if $(G,\rho,\ell)$ is a local limit of finite geometric graphs (as in \cref{def: local weighted limit}), then its law is unimodular. The latter fact in the unweighted case was first observed by Benjamini and Schramm \cite{Benjamini_Schramm2001local_limit}. We will use this fact in the proof of \cref{thm: unweight local limit for regular graph sequence}.


\subsection{A theorem of Aldous and Steele \cite{Aldous_Steele2004objective_method}: local limit of MSTs}

Given a finite, connected graph $G=(V,E)$ and an injective weight function $\ell:E\to(0,\infty)$, there is a unique spanning tree $T_\ell$ of $G$ that has the minimal sum of weights, $\sum_{e\in E(T_\ell)}\ell(e)$ among all spanning trees of $G$ (\cite[Exercise 11.1]{LP2016}). We call the tree $T_\ell$ the  \textit{minimal spanning tree} (MST) of the weighted graph $(G,\ell)$  and  write it  as $\mathrm{MST}(G,\ell)$. Note that the MST may inherit the weights from $G$ so it can be viewed as an element of $\Gb$ or $\Gs$.


The following theorem due to Aldous and Steele \cite[Theorem 5.4]{Aldous_Steele2004objective_method} provides a ``continuity'' property for the MST, \emph{under the weighted topology} of $\Gs$, that holds in great generality. It roughly states that if $(G_n, \ell_n) \stackrel[n \to \infty]{d}{\longrightarrow} (G,\rho,\ell)$, then the corresponding MST of $(G_n,\ell_n)$ converges locally to the component of $\rho$ in the WMSF of $(G,\rho,\ell)$ (recall the definition of the WMSF in \cref{sec: intro: MSF on PWIT}).

\begin{theorem}[Aldous-Steele \cite{Aldous_Steele2004objective_method}]\label{thm: AS2004}
	Let $(G_\infty,\rho,\ell)$ denote a $\Gs$-valued random variable such that with probability one $G_\infty$ has infinitely many vertices and no two of the edges have the same weight.  Further, let $\big\{(G_n,\ell_n)\}_{n\geq 1}$ be a sequence of $\Gs$-valued random variables such that almost surely for each $n$ the graph $G_n$ is finite, $\ell_n$ is injective, and $|V(G_n)|\to \infty$. If 
	\be\label{eq: 5.2}
	(G_n,\ell_n)\stackrel[n \to \infty]{d}{\longrightarrow}(G_\infty,\rho,\ell) \, ,
	\ee
	then one has the joint weak convergence in $\Gs\times \Gs$
	\be\label{eq: 5.3}
	\Big((G_n,\ell_n),\big( \mathrm{MST}(G_n, \ell_n),\ell_n\big)\Big)\stackrel[n \to \infty]{d}{\longrightarrow}\Big((G_\infty,\rho,\ell),\big(\mathrm{WMSF}(G_\infty,\ell),\rho,\ell\big)\Big)\,.
	\ee
	
	\noindent Furthermore, if $N_n$ denotes the degree of the root in $ \mathrm{MST}(G_n, \ell_n)$ and $N$ denotes the degree of the root in  $\mathrm{WMSF}(G_\infty,\ell)$, then
	\be\label{eq: 5.4}
	N_n\stackrel{d}{\longrightarrow} N\,\,\,\textnormal{ and }\,\ \mathbb{E}[N_n]\to\mathbb{E}[N]=2 \, .
	\ee
	Lastly, if $L_n$ denotes the sum of lengths of the edges incident to the root of $ \mathrm{MST}(G_n, \ell_n)$ and $L$ denotes the corresponding quantity for $\mathrm{WMSF}(G_\infty,\ell)$, then
	\be\label{eq: 5.5}
	L_n\stackrel{d}{\longrightarrow}L\,\,\,\textnormal{ as }\,\,n\to\infty.
	\ee
\end{theorem}



\subsection{From weighted to unweighted local convergence}

Our goal in this section is to show that under the same setting and assumptions as \cref{thm: AS2004} one can additionally conclude convergence under the \emph{unweighted topology}, i.e., that of $\Gb$. 

\begin{theorem}\label{thm:WeightedUnweighted} Assume that $(H_n, \ell_n)$ are finite (possibly random) geometric graphs and that $(H, \rho, \ell)$ is a $\Gs$-valued random variable such that
	$$ (H_n, \ell_n) \stackrel[n \to \infty]{d}{\longrightarrow} (H,\rho,\ell) \, .$$
	Furthermore, let $\rho_n$ be a uniformly chosen vertex of $H_n$ and write $N_n$ and $N$ for the degree of $\rho_n$ and of $\rho$ respectively. Denote by $L_n$ and $L$ the weighted degrees, that is, $L_n = \sum_{e : \rho_n \in e} \ell_n(e)$ and $L = \sum_{e: \rho\in e} \ell(e)$. If 
	$$ N_n \stackrel{d}{\longrightarrow} N \quad \mathrm{and} \quad L_n \stackrel{d}{\longrightarrow} L \quad \mathrm{and} \quad N<\infty \,\, \mathrm{a.s.}\, ,$$
	then 
	$$ H_n \stackrel[n \to \infty]{d}{\longrightarrow} (H,\rho) \, .$$
\end{theorem}

The proof of \cref{thm:WeightedUnweighted} is postponed to the next section. Next, we provide a (probably well-known) lemma showing that in the setting of \cref{thm: unweight local limit for regular graph sequence}, the weighted local limit of the ambient graph $(G_n, \ell_n)$ is the PWIT. This was known for the complete graph on $n$ vertices with exponential distributed weights \cite[Theorem 4.1]{Aldous_Steele2004objective_method}. 


\begin{lemma}\label{lem: convergence to PWIT by weighted regular graph sequences}
	Suppose $\{G_n\}$ is a sequence of finite, simple, connected, regular graphs with degree $\deg(G_n)\to\infty$. 
	For each $n$, assign i.i.d.\ weights $\ell_n$ to the edges of $G_n$ using uniform distribution on $[0,\deg(G_n)]$.
	%
	Draw the root $\rho_n$ uniformly at random from $V(G_n)$ and independently of each other. 
	Then the random rooted weighted graphs $(G_n,\rho_n,\ell_n)$ converge in distribution to the PWIT.
\end{lemma}
\begin{proof}[Proof of Lemma \ref{lem: convergence to PWIT by weighted regular graph sequences}] 
	For positive integers $a,b\geq1$, let $\mathscr{N}_{a,b}(U,\emptyset,W)$ be the minimal $a$-ary subtree of height $b$ in the PWIT $(U,\emptyset,W)$; i.e., it is the finite weighted tree with vertex set $\{\emptyset\}\cup \{(i_1,\ldots,i_k)\colon 1\leq i_j\leq a \textnormal{ and }k\leq b\}$. For a regular graph $(G,\rho,\ell)\in \Gs$ with distinct edge weights and degree $\deg(G)\geq a^{b+1}>1+a+\cdots+a^b$, one can define a similar 
	minimal $a$-ary subtree $\mathscr{N}_{a,b}(G,\rho,\ell)$ of height $b$ inside $(G,\rho,\ell)$ with vertex set $\{v_\emptyset \}\cup \{v_{i_1,\ldots,i_k}\colon k\leq b,1\leq i_j\leq a\}$ determined by the following exploration process:
	\begin{enumerate}
		\item Initially let $v_\emptyset:=\rho$, the set of explored vertices $\mathscr{E}$  be empty and the set of active vertices $\mathscr{A}=\{v_\emptyset\}$.
		
		\item Explore the edges incident to $v_\emptyset$ in increasing order of weights and stop after exploring the first $a$ such edges. Let $v_1,\ldots,v_a$ be the other endpoints of these edges in that order. Add $v_1,\ldots,v_a$ to $\mathscr{A}$ and $\mathscr{E}$. Then delete $v_\emptyset$ from $\mathscr{A}$.

		\item If $\mathscr{A}$ is not empty, let $v_{i_1,\ldots,i_k}$ be the  vertex in $\mathscr{A}$ with the smallest index in the lexicographic order. If $k\geq b$, then we stop the entire exploration process. Otherwise, explore the edges from  $v_{i_1,\ldots,i_k}$ to $V(G)\setminus \mathscr{E}$ in  increasing order of weights and stop after exploring $a$ such edges. Let $v_{i_1,\ldots,i_k,1},\cdots,v_{i_1,\ldots,i_k,a}$ be the other endpoints in that order and add these $a$ vertices to $\mathscr{A}$ and $\mathscr{E}$. Delete $v_{i_1,\ldots,i_k}$ from $\mathscr{A}$. Repeat this step until we stop the entire exploration process. 
		
	\end{enumerate}
	We write $\Phi_n:\mathscr{N}_{a,b}(U,\emptyset,W)\to \mathscr{N}_{a,b}(G_n,\rho_n,\ell_n)$ for the natural map given by $\Phi_n(\emptyset)=\rho_n$ and $\Phi_n\big((i_1,\ldots,i_k)\big)=v_{i_1,\ldots,i_k}$. We first observe that for any integers $a,b\geq1$, with probability $1-o(1)$ one can couple $\mathscr{N}_{a,b}(G_n,\rho_n,\ell_n)$ and $\mathscr{N}_{a,b}(U,\emptyset,W)$ so that $\big|W(e)-\ell_n(\Phi_n(e)) \big|=o(1)$ for all $e\in \mathscr{N}_{a,b}(U,\emptyset,W)$. Indeed, this follows immediately from an observation from \cite[Section 1.1]{Addario_Griffiths_Kang2012invasion_percolation_PWIT}: if $\mathbf{X}=\big(X_1^*,\ldots,X_n^*\big)$ denotes the order statistics of $n$ independent uniform$[0,n]$ random variables, then the vector $\big(X_1^*,\ldots,X_{\lfloor \sqrt{n} \rfloor }^*\big)$ has total variation distance $O(n^{-1/2})$ (we only need here that it is $o(1)$) from the first $\lfloor \sqrt{n} \rfloor$ points of a homogeneous rate one Poisson process on $[0,\infty)$.
	Secondly we show that for any $T>0$, one has that 
	\be\label{eq:nolightcycles}
	\lim_{n\to\infty}\mathbb{P}(C_n\leq T)=0,
	\ee
	where $C_n$ is the minimal weighted length of a cycle containing the root. Indeed, by the first observation we may choose $a,b$ large enough so that with probability at least $1-\eps$, any path from $\rho_n$ to the leaves of the tree $\mathscr{N}_{a,b}(G_n,\rho_n,\ell_n)$ has weight at least $T$ and any edge not belonging to $\mathscr{N}_{a,b}(G_n,\rho_n,\ell_n)$ that has precisely one vertex in $\mathscr{N}_{a,b}(G_n,\rho_n,\ell_n)$ that is \emph{not} a leaf (at height $b$) has weight at least $T$. Furthermore, conditioned on $\mathscr{N}_{a,b}(G_n,\rho_n,\ell_n)$, any edge of $G_n$ between two vertices of this tree that is not in the tree has not exposed its weights so the probability that it has weight at least $T$ is $1-T/\deg(G_n)$. Since there are only finitely many such edges ($G_n$ is simple), depending only on the choice of $a$ and $b$, we can also deduce that with probability at least $1-\eps$ all such edges have weight at least $T$. 
	
	On these two events, any cycle containing $\rho_n$ must have weight at least $T$. Indeed, either the cycle contains a path from $\rho_n$ to the leaves of $\mathscr{N}_{a,b}(G_n,\rho_n,\ell_n)$ or it ``exits'' $\mathscr{N}_{a,b}(G_n,\rho_n,\ell_n)$ before reaching the leaves and therefore must traverse an edge with weight at least $T$. This establishes \eqref{eq:nolightcycles} and  together with the observation preceding it our proof is concluded.
\end{proof}

With these three ingredients in place (\cref{thm: AS2004}, \cref{thm:WeightedUnweighted} and Lemma \ref{lem: convergence to PWIT by weighted regular graph sequences}) our main result \cref{thm: unweight local limit for regular graph sequence} easily follows.

\begin{proof}[Proof of \cref{thm: unweight local limit for regular graph sequence}]
	Since multiplying a same constant to all the edge weights does not change the  MST, we can use i.i.d.\ uniform$[0,\deg(G_n)]$ weights $\ell_n$ to sample $M_n$.
	Lemma \ref{lem: convergence to PWIT by weighted regular graph sequences} immediately shows that assumption \eqref{eq: 5.2} holds. The other assumptions of \cref{thm: AS2004} hold almost surely by definition. We therefore obtain that \eqref{eq: 5.3}, \eqref{eq: 5.4} and \eqref{eq: 5.5} hold. Thus the assumptions of \cref{thm:WeightedUnweighted} hold with $(H_n, \ell_n) = \mathrm{MST}(G_n, \ell_n)$. The conclusion of \cref{thm:WeightedUnweighted} finishes the proof.
\end{proof}

\subsection{Proof of Theorem \ref{thm:WeightedUnweighted}}
Before the proof, let us consider some examples showing that the conditions of the theorems are necessary and that one cannot simply argue by ``forgetting'' the weights. 

\begin{example}[Unimodularity is necessary] Consider two disjoint copies of the complete graph on $n$ vertices connected by a single edge. Put i.i.d.~weights of exponential distribution with mean $n$ on all the edges. If we root the graphs in one of the two vertices of the special edge connecting the two complete graphs, then the unweighted local limit seen from this vertex is two disjoint independent WMSF on the PWIT with their roots connected by an edge. The weighted local limit, however, is just a copy of the WMSF on the PWIT, since the special edge has weight of order $n$ with high probability. 
\end{example}

\begin{example}[$L_n \stackrel{d}{\longrightarrow} L$ is necessary] Consider two disjoint cycles of length $n$ connected by a matching; put i.i.d.~uniform $[0,1]$ weights on the two cycles, and weight $n$ on the edges of the matching. The weighted local limit is then clearly a copy of $\mathbb{Z}$ with i.i.d.~uniform $[0,1]$ weights while the unweighted limit is two copies of $\mathbb{Z}$ connected by a matching.
\end{example}

For simplicity we use the following notations throughout the proof of \cref{thm:WeightedUnweighted}:
$$B_n(r):=B_{H_n}(\rho_n,r)\, , \enspace B_n^W(t):=B_{H_n}^W(\rho_n,t)\, , \enspace B_\infty(r):=B_{H}(\rho,r)\, , \enspace B_\infty^W(t):=B_{H}^W(\rho,t) \, ,
$$
where we use the letter $r$ for the the unweighted distance and $t$ for the weighted one.

The overall strategy is simple: we will show that with high probability unweighted balls of radius $r$ are contained in weighted balls of some large weighted radius $T$. By the weighted convergence, with high probability $B_n^W(T)$ is isomorphic to $B_\infty^W(T)$ (with close edge weights) which implies that the unweighted balls are isomorphic as well. Technically it will amount to a few applications of the mass-transport principle \eqref{eq: MTP}. We first prove some preparatory lemmas.
\begin{lemma}\label{lem: locally finite of the limiting graph}
	All vertex degrees of the limiting graph $H$ in Theorem \ref{thm:WeightedUnweighted} are a.s.~finite. 
\end{lemma}
\begin{proof}
	Consider the mass transport in which
	each vertex of $H$ sends unit mass to every vertex of $H$ of infinite degree. Since $N<\infty$, the expected mass entering the root is zero. By the mass-transport principle \eqref{eq: MTP} the expected mass sent from the root is also zero, so there are no vertices of infinite degree almost surely.
\end{proof}

\begin{lemma} \label{lem:BallsCannotTouchLargeEdges} For any $\eps>0$ and $t>0$ there exists $T=T(\eps,t)\geq t$ such that for all $n\geq 1$ the probability that some vertex of $B_n^W(t)$ is incident to an edge with weight at least $T$ is at most $\eps$. 
\end{lemma}
\begin{proof} 
	By Skorohod's representation theorem (\cite[Theorem 3.1.1]{Skorokhod1956limit} or \cite[Theorem 6.7]{Bill99}), we may suppose the $\Gs$-valued r.v.'s $ \big(H_n,\rho_n,\ell_n\big)$ and $\big(H,\rho,\ell\big)$ are all defined on a common probability space and that almost surely
	\[
	\big(H_n,\rho_n,\ell_n\big)\to \big(H,\rho,\ell\big) \textnormal{ as }n\to\infty.
	\]
	Since the set of non-continuity points for a fixed realization of $\big(H,\rho,\ell\big)$ has no accumulation points, one can pick a random and measurable continuity point $T=T(\omega) \geq 2t$ of $\big(H,\rho,\ell\big)$ (for instance, if $2t$ is a continuity point, then let $T=2t$, otherwise let $T$ be the average of $2t$ and the smallest non-continuity point larger than $2t$). By the definition of convergence in $\Gs$ there exists a random integer $K=K(\omega)>0$ such that for all $n\geq K$ there is a root-preserving isomorphism $\phi_n$ from  $B_{\infty}^W(T)$ onto $B_{n}^W(T)$. Hence the sequence  $\{|B_n^W(2t)|\}_{n\geq 1}$ is bounded by a sequence of random variables $\{|B_n^W(T)|\}_{n\geq 1}$ that almost surely converges to  $|B_{\infty}^W(T)|<\infty$. Thus $\{|B_n^W(2t)|\}_{n\geq 1}$ is a tight sequence, so there is some $M=M(\eps,t)$ such that $\sup_n \P(|B_n^W(2t)|> M)\leq \eps/2$. Similarly $\{L_n\}$ is tight so there exists $T=T(\eps,t)$ such that $\sup_n \P(L_n\geq T)\leq \eps/2M$.


	We now define a mass transport (note that here the mass-transport principle is unnecessary since the graphs are finite; we use the mass transport terminology instead of double counting for convenience). Each vertex $u$ sends unit mass to a vertex $v$ if $d_W(u,v) \leq t$ and $v$ is incident to an edge $e$ of weight $\ell_n(e) \geq T$ and $|B_{H_n}^W(u,2t)|\leq M$; otherwise zero mass is sent to $v$. Positive mass enters $\rho_n$ only when $L_n \geq T$ and only from vertices $u \in B_{H_n}^W(\rho_n,t)$ such that $|B_{H_n}^W(u,2t)|\leq M$. Hence the expected mass entering $\rho_n$ is at most $M \P( L_n \geq T)\leq \eps/2$. By the mass-transport principle \eqref{eq: MTP}, the expected mass sent from $\rho_n$ is at most $\eps/2$. On the other hand it is at least the probability that $|B_n^W(2t)|\leq M$ and that there exists a vertex of $B_n^W(t)$ that is incident to an edge of weight at least $T$. Since $\P(|B_n^W(2t)|\leq M)\geq 1-\eps/2$, we deduce that the probability that some vertex of $B_n^W(t)$ is incident to an edge with weight at least $T$ is at most $\eps$. 
\end{proof}
\begin{lemma}\label{lem: unweighted balls contained in large weighted balls}
	For any integer $r>0$ and $\varepsilon>0$, there exists a real number $T=T(r,\varepsilon)>0$ such that 
	$$\inf_n \mathbb{P}\big[  B_n(r)\subset B_n^W(T) \big]\geq 1-\varepsilon \, .
	$$
\end{lemma}
\begin{proof} We apply  Lemma~\ref{lem:BallsCannotTouchLargeEdges} $r$ times (with $\eps = \varepsilon/r$) to obtain real numbers $T(0,\eps) \leq T(1,\eps)\leq \ldots \leq T(r,\eps)$ depending only on $r$ and $\eps$ such that $T(0,\eps)=0$ and with probability at least $1-\eps$, for all $i=0,\ldots,r-1$, the vertices of the ball $B_n^W(T(i,\eps))$ are not incident to any edge of weight at least $T(i+1,\eps)$. On this event, $B_n(r)\subset B_n^W(T(r,\eps))$. 
\end{proof}


We now prove Theorem \ref{thm:WeightedUnweighted}. 

\begin{proof}[Proof of \cref{thm:WeightedUnweighted}]
	By Lemma \ref{lem: locally finite of the limiting graph} for any $r$ the ball $B_\infty(r)$ is almost surely finite. Hence the maximal weight of a simple path in the ball is an almost surely finite random variable, so for any $\varepsilon>0$ there exists $T_1=T_1(r,\varepsilon)<\infty$ such that
	\be\label{eq: choice of large T}
	\mathbb{P}\Big[B_{\infty}(r)\subset B_{\infty}^W(T_1) \Big]\geq 1-\varepsilon \, .
	\ee
	By Lemma \ref{lem: unweighted balls contained in large weighted balls} there exist $T_2=T_2(r,\eps)$ such that for all $n\geq 1$
	\be\label{eq: choice of large T_1}
	\mathbb{P}\Big[B_{n}(r)\subset B_{n}^W(T_2) \Big]\geq 1-\varepsilon \, .
	\ee
	
	%
	
	Using Skorohod's representation theorem as in the proof of Lemma \ref{lem:BallsCannotTouchLargeEdges} again, one can pick a random and measurable continuity point $T_3=T_3(\omega) \geq \max(T_1,T_2)$ of $\big(H,\rho,\ell\big)$. By the definition of  convergence in $\Gs$ there exists a random integer $K_1=K_1(\omega)>0$ such that for all $n\geq K_1$ there is a root-preserving isomorphism $\phi_n$ from  $B_{\infty}^W(T_3)$ onto $B_{n}^W(T_3)$. Choose a large constant $K>0$ so that $\mathbb{P}[K_1\geq K]\leq \varepsilon$. Then, by \eqref{eq: choice of large T} and \eqref{eq: choice of large T_1}, for any $n\geq K$, with probability at least $1-3\eps$, the unweighted balls $B_{n}(r)$
	and $B_{\infty}(r)$ are subgraphs of $B_{n}^W(T_3)$ and $B_{\infty}^W(T_3)$, respectively. On this event, the restriction of the isomorphism $\phi_n$ gives an isomorphism from 
	$B_{\infty}(r)$ to $B_{n}(r)$. We conclude that for any $n \geq K$, the probability that $B_n(r)$ is isomorphic to $B_\infty(r)$ is at least $1-3\eps$, concluding the proof.
\end{proof}

\begin{acknowledgements} This research is supported by ERC consolidator grant 101001124 (UniversalMap), and by ISF grant 1294/19. 
	The authors would like to thank the referee for careful reading and many helpful comments. 
	
\end{acknowledgements}

\printnomenclature[2cm]

\bibliography{spec_diff_ref}

\begin{bibdiv}
\begin{biblist}

\bib{Addario-Berry_Broutin_Reed2009diamater_MST}{article}{
      author={Addario-Berry, L.},
      author={Broutin, N.},
      author={Reed, B.},
       title={Critical random graphs and the structure of a minimum spanning
  tree},
        date={2009},
        ISSN={1042-9832},
     journal={Random Structures Algorithms},
      volume={35},
      number={3},
       pages={323\ndash 347},
         url={https://doi.org/10.1002/rsa.20241},
      review={\MR{2548517}},
}

\bib{Addario2013local_limit_MST_complete_graphs}{article}{
      author={Addario-Berry, Louigi},
       title={The local weak limit of the minimum spanning tree of the complete
  graph},
        date={2013},
     journal={arXiv preprint arXiv:1301.1667},
}

\bib{Addario-Berry_etal2017_scaling_limit}{article}{
      author={Addario-Berry, Louigi},
      author={Broutin, Nicolas},
      author={Goldschmidt, Christina},
      author={Miermont, Gr\'{e}gory},
       title={The scaling limit of the minimum spanning tree of the complete
  graph},
        date={2017},
        ISSN={0091-1798},
     journal={Ann. Probab.},
      volume={45},
      number={5},
       pages={3075\ndash 3144},
         url={https://doi.org/10.1214/16-AOP1132},
      review={\MR{3706739}},
}

\bib{Addario_Griffiths_Kang2012invasion_percolation_PWIT}{article}{
      author={Addario-Berry, Louigi},
      author={Griffiths, Simon},
      author={Kang, Ross~J.},
       title={Invasion percolation on the {P}oisson-weighted infinite tree},
        date={2012},
        ISSN={1050-5164},
     journal={Ann. Appl. Probab.},
      volume={22},
      number={3},
       pages={931\ndash 970},
         url={https://doi.org/10.1214/11-AAP761},
      review={\MR{2977982}},
}

\bib{AL2007}{article}{
      author={Aldous, David},
      author={Lyons, Russell},
       title={Processes on unimodular random networks},
        date={2007},
        ISSN={1083-6489},
     journal={Electron. J. Probab.},
      volume={12},
       pages={no. 54, 1454\ndash 1508},
         url={https://doi.org/10.1214/EJP.v12-463},
      review={\MR{2354165}},
}

\bib{Aldous_Steele2004objective_method}{incollection}{
      author={Aldous, David},
      author={Steele, J.~Michael},
       title={The objective method: probabilistic combinatorial optimization
  and local weak convergence},
        date={2004},
   booktitle={Probability on discrete structures},
      series={Encyclopaedia Math. Sci.},
      volume={110},
   publisher={Springer, Berlin},
       pages={1\ndash 72},
         url={https://doi.org/10.1007/978-3-662-09444-0_1},
      review={\MR{2023650}},
}

\bib{BCK2005sub_Gaussian}{article}{
      author={Barlow, Martin~T.},
      author={Coulhon, Thierry},
      author={Kumagai, Takashi},
       title={Characterization of sub-{G}aussian heat kernel estimates on
  strongly recurrent graphs},
        date={2005},
        ISSN={0010-3640},
     journal={Comm. Pure Appl. Math.},
      volume={58},
      number={12},
       pages={1642\ndash 1677},
         url={https://doi.org/10.1002/cpa.20091},
      review={\MR{2177164}},
}

\bib{BJKS2008RW_IIC_oriented_percolation}{article}{
      author={Barlow, Martin~T.},
      author={J\'{a}rai, Antal~A.},
      author={Kumagai, Takashi},
      author={Slade, Gordon},
       title={Random walk on the incipient infinite cluster for oriented
  percolation in high dimensions},
        date={2008},
        ISSN={0010-3616},
     journal={Comm. Math. Phys.},
      volume={278},
      number={2},
       pages={385\ndash 431},
         url={https://doi.org/10.1007/s00220-007-0410-4},
      review={\MR{2372764}},
}

\bib{BarlowKumagaiTrees06}{article}{
      author={Barlow, Martin~T.},
      author={Kumagai, Takashi},
       title={Random walk on the incipient infinite cluster on trees},
        date={2006},
        ISSN={0019-2082},
     journal={Illinois J. Math.},
      volume={50},
      number={1-4},
       pages={33\ndash 65},
         url={http://projecteuclid.org/euclid.ijm/1258059469},
      review={\MR{2247823}},
}

\bib{Benjamini_Schramm2001local_limit}{article}{
      author={Benjamini, Itai},
      author={Schramm, Oded},
       title={Recurrence of distributional limits of finite planar graphs},
        date={2001},
        ISSN={1083-6489},
     journal={Electron. J. Probab.},
      volume={6},
       pages={no. 23, 13},
         url={https://doi.org/10.1214/EJP.v6-96},
      review={\MR{1873300}},
}

\bib{Bill99}{book}{
      author={Billingsley, Patrick},
       title={Convergence of probability measures},
     edition={Second},
      series={Wiley Series in Probability and Statistics: Probability and
  Statistics},
   publisher={John Wiley \& Sons, Inc., New York},
        date={1999},
        ISBN={0-471-19745-9},
         url={https://doi.org/10.1002/9780470316962},
        note={A Wiley-Interscience Publication},
      review={\MR{1700749}},
}

\bib{Durrett_book_5th_edition}{book}{
      author={Durrett, Rick},
       title={Probability---theory and examples},
      series={Cambridge Series in Statistical and Probabilistic Mathematics},
   publisher={Cambridge University Press, Cambridge},
        date={2019},
      volume={49},
        ISBN={978-1-108-47368-2},
         url={https://doi.org/10.1017/9781108591034},
        note={Fifth edition of [ MR1068527]},
      review={\MR{3930614}},
}

\bib{Frieze1985value_MST}{article}{
      author={Frieze, A.~M.},
       title={On the value of a random minimum spanning tree problem},
        date={1985},
        ISSN={0166-218X},
     journal={Discrete Appl. Math.},
      volume={10},
      number={1},
       pages={47\ndash 56},
         url={https://doi.org/10.1016/0166-218X(85)90058-7},
      review={\MR{770868}},
}

\bib{Kumagai_Misumi2008}{article}{
      author={Kumagai, Takashi},
      author={Misumi, Jun},
       title={Heat kernel estimates for strongly recurrent random walk on
  random media},
        date={2008},
        ISSN={0894-9840},
     journal={J. Theoret. Probab.},
      volume={21},
      number={4},
       pages={910\ndash 935},
         url={https://doi.org/10.1007/s10959-008-0183-5},
      review={\MR{2443641}},
}

\bib{LP2016}{book}{
      author={Lyons, Russell},
      author={Peres, Yuval},
       title={Probability on trees and networks},
      series={Cambridge Series in Statistical and Probabilistic Mathematics},
   publisher={Cambridge University Press, New York},
        date={2016},
      volume={42},
        ISBN={978-1-107-16015-6},
         url={http://dx.doi.org/10.1017/9781316672815},
        note={Available at \url{http://rdlyons.pages.iu.edu/}},
      review={\MR{3616205}},
}

\bib{Lyons_Peres_Schramm2006minimal_spanning_forests}{article}{
      author={Lyons, Russell},
      author={Peres, Yuval},
      author={Schramm, Oded},
       title={Minimal spanning forests},
        date={2006},
        ISSN={0091-1798},
     journal={Ann. Probab.},
      volume={34},
      number={5},
       pages={1665\ndash 1692},
         url={https://doi.org/10.1214/009117906000000269},
      review={\MR{2271476}},
}

\bib{Nachmias_Peres2021local_limit}{article}{
      author={Nachmias, Asaf},
      author={Peres, Yuval},
       title={The local limit of uniform spanning trees},
        date={2022},
        ISSN={0178-8051},
     journal={Probab. Theory Related Fields},
      volume={182},
      number={3-4},
       pages={1133\ndash 1161},
         url={https://doi.org/10.1007/s00440-021-01072-2},
      review={\MR{4408511}},
}

\bib{Skorokhod1956limit}{article}{
      author={Skorokhod, Anatoly~V},
       title={Limit theorems for stochastic processes},
        date={1956},
     journal={Theory of Probability \& Its Applications},
      volume={1},
      number={3},
       pages={261\ndash 290},
}

\bib{van_den_Berg_etal2007size_of_pond_IPC}{article}{
      author={van~den Berg, Jacob},
      author={J\'{a}rai, Antal~A.},
      author={V\'{a}gv\"{o}lgyi, B\'{a}lint},
       title={The size of a pond in 2{D} invasion percolation},
        date={2007},
        ISSN={1083-589X},
     journal={Electron. Comm. Probab.},
      volume={12},
       pages={411\ndash 420},
         url={https://doi.org/10.1214/ECP.v12-1327},
      review={\MR{2350578}},
}

\bib{Hofstad2017book_volume1}{book}{
      author={van~der Hofstad, Remco},
       title={Random graphs and complex networks. {V}ol. 1},
      series={Cambridge Series in Statistical and Probabilistic Mathematics,
  [43]},
   publisher={Cambridge University Press, Cambridge},
        date={2017},
        ISBN={978-1-107-17287-6},
         url={https://doi.org/10.1017/9781316779422},
      review={\MR{3617364}},
}

\end{biblist}
\end{bibdiv}
\bibliographystyle{alpha}

\end{document}